\documentclass{article}
\pdfoutput=1

\usepackage[utf8]{inputenc}
\usepackage{geometry}
\usepackage{amsmath}
\usepackage{amsfonts}
\usepackage{amssymb}
\usepackage{amsthm}
\usepackage{graphicx}
\usepackage[dvipsnames]{xcolor}
\usepackage{tikz}
\usetikzlibrary{decorations.pathreplacing}
\usepackage{enumitem,mdwlist}
\setlist[itemize]{topsep=0ex,itemsep=0ex,parsep=0.4ex}
\setlist[enumerate]{topsep=0ex,itemsep=0.1ex,parsep=0.4ex}
\usepackage{bm}
\usepackage{caption}
\usepackage{subcaption}
\usepackage{url}%
\usepackage{mathtools}
\usepackage{systeme}
\usepackage[linktoc = all, hidelinks, colorlinks, unicode=true, pagebackref=true]{hyperref} 
\hypersetup{linkcolor={blue}, citecolor={green!60!black}, urlcolor={black}}
\usepackage[capitalise]{cleveref}
\usepackage{thm-restate}
\usepackage{comment}\usepackage[dvipsnames]{xcolor}
\newtheorem{theorem}{Theorem}[section]
\newtheorem{definition}[theorem]{Definition}
\newtheorem{proposition}[theorem]{Proposition}

\newtheorem{conjecture}[theorem]{Conjecture}
\newtheorem{lemma}[theorem]{Lemma}

\newtheorem{observation}[theorem]{Observation}
\newtheorem{fact}[theorem]{Fact}
\newtheorem*{remarks*}{Remarks}
\newtheorem*{remark*}{Remark}
\newtheorem{claim}[theorem]{Claim}
\crefname{equation}{}{}
\crefname{item}{}{}
\crefname{item}{}{}
\numberwithin{equation}{section}
\newenvironment{proofclaim}[1][Proof of claim]{\begin{proof}[#1]}{\end{proof}}
\newcommand{\eps}{\varepsilon}

\newcommand{\N}{\mathbb{N}}
\newcommand{\taux}{T_{\text{aux}}}

\newcommand{\var}{h}
\renewcommand{\geq}{\geqslant}
\renewcommand{\leq}{\leqslant}
\DeclarePairedDelimiter{\ceil}{\lceil}{\rceil}

\newcommand{\defn}[1]{\textcolor{red!60!black}{\emph{#1}}}

\DeclareMathOperator{\gs}{gs}

\title{On the gracesize of trees}

\author{Shoham Letzter\thanks{Department of Mathematics, University College London, Gower Street, WC1E 6BT, UK. Emails: \tt{\{\href{mailto:s.letzter@ucl.ac.uk}{s.letzter},\href{mailto:a.pokrovskiy@ucl.ac.uk}{a.pokrovskiy},\href{mailto:ella.williams.23@ucl.ac.uk}{ella. williams.23}\}@ucl.ac.uk}.} \thanks{Research supported by the Royal Society.} \and Alexey Pokrovskiy\footnotemark[1] \and Ella Williams\footnotemark[1] \thanks{Research supported by the Martingale Foundation.}}
\date{}

\begin{document}

\setlength{\baselineskip}{14.8pt}

\maketitle

\begin{abstract}

\setlength{\baselineskip}{13pt}
    \setlength{\parskip}{\medskipamount}
    \setlength{\parindent}{0pt}
    \noindent
    
An $n$-vertex tree $T$ is said to be \textit{graceful} if there exists a bijective labelling $\phi:V(T)\to \{1,\ldots,n\}$ such that the edge-differences $\{|\phi(x)-\phi(y)| : xy\in E(T)\}$ are pairwise distinct. The longstanding graceful tree conjecture, posed by R\'{o}sa in the 1960s, asserts that every tree is graceful.
The \textit{gracesize} of an $n$-vertex tree $T$, denoted $\gs(T)$, is the maximum possible number of distinct edge-differences over all bijective labellings $\phi:V(T)\to \{1,\ldots,n\}$. The graceful tree conjecture is therefore equivalent to the statement that $\gs(T)=n-1$ for all $n$-vertex trees.

We prove an asymptotic version of this conjecture by showing that for every $\varepsilon>0$, there exists $n_0$ such that every tree on $n>n_0$ vertices satisfies $\gs(T)\ge (1-\varepsilon)n$. In other words, every sufficiently large tree admits an almost graceful labelling.
\end{abstract}

\section{Introduction}
 Within extremal combinatorics, a common question to pose is ``which properties must a host graph possess in order to guarantee that it contains a given subgraph?'', with classical results such as theorems of Mantel \cite{Mantel} and Dirac \cite{dirac1952some} providing sufficient conditions for the existence of triangles and spanning cycles respectively. Colouring analogues of these problems have been extensively studied, most notably in Ramsey theory, which investigates the appearance of monochromatic subgraphs in edge-coloured graphs. A complementary perspective is given by anti-Ramsey theory, initiated by Erd\H{o}s, Simonovits, and S\'{o}s \cite{Erdos1975antiramsey}, which concerns the existence of rainbow subgraphs, in which all edges must receive distinct colours. 
In this paper, we use this colouring framework to study a classical labelling problem in combinatorics: the graceful tree conjecture. Specifically, we tackle the problem by reformulating the conjecture in terms of finding rainbow trees in appropriately edge-coloured complete graphs, connecting graceful labellings to tools from extremal and probabilistic graph theory.

     A \defn{labelling} of a graph $G$ on $m$ edges is an injective mapping $\phi : V(G) \rightarrow \mathbb{N}$, and we say $\phi$ is \defn{graceful} if its image is $\{1,\ldots,m+1\}$ and the values of $|\phi(x) - \phi(y)|$ are pairwise distinct over all edges $xy \in E(G)$.  A graph is said to be \defn{graceful} if it admits a graceful labelling.
Graceful labellings were first introduced by R\'{o}sa \cite{rosa1966} in 1966 (as $\beta$-valuations) and later renamed by Golomb \cite{GOLOMB197223}. Given a labelling $\phi$, we call the values of $|\phi(x) - \phi(y)|$ for $xy \in E(G)$ the \defn{edge-differences} of $\phi$. The main theme of research in this area is to determine which graphs are graceful. Since any graceful graph on $n$ vertices necessarily satisfies $e(G)\leq n-1$, almost all graphs are not graceful simply because they are too dense, as noted by Erd\H{o}s (unpublished, see \cite{GOLOMB197223,graham_sloane1980}). Trees are an interesting class of graphs to consider, since the set of edge-differences of any graceful labelling $\phi$ of an $n$-vertex tree must exactly coincide with the set $\{1,\dots,n-1\}$.
This resulted in the graceful tree conjecture, posed in the 1960s, and usually accredited to R\'{o}sa.
\begin{conjecture}[Graceful tree conjecture]\label{conj:gtc}
    Every tree admits a graceful labelling.
\end{conjecture}

Graceful labellings were initially motivated as a method for proving Ringel's conjecture \cite{Ringel1963}, a notable conjecture from 1963 which proposed that for every $(n+1)$-vertex tree $T$, the complete graph $K_{2n+1}$ can be decomposed into edge-disjoint copies of $T$.  \Cref{conj:gtc} implies a stronger variant of Ringel's conjecture, attributed to Kotzig, suggesting that this decomposition can be done by cyclically shifting the copies of the tree. 
Indeed, if there exists a graceful labelling of an $(n+1)$-vertex tree $T$, say $\phi:V(T) \rightarrow \{1,\ldots,n+1\}$, then one can construct an embedding of $T$ in $K_{2n+1}$ using only the first $n+1$ vertices, by embedding each vertex of $T$ to its image under $\phi$. Cyclically shifting this copy of $T$ $2n+1$ times by adding $1$~modulo~$2n+1$ to each label at each shift, we get $2n+1$ copies of $T$, which can be seen to be pairwise edge-disjoint, and which cover all edges of $K_{2n+1}$ using the properties of the graceful labelling $\phi$. Ringel's conjecture (and its stronger variant) was recently proven for all sufficiently large $n$ by Montgomery, Pokrovskiy and Sudakov \cite{montgomery2021ringel} and independently Keevash and Staden \cite{keevash2025ringel}. The strength of the graceful tree conjecture illustrates that it stands to be a challenging and interesting problem in itself. 

Over the years, research towards \cref{conj:gtc} has been fruitful, yet in full generality it remains open. Despite a resolution for certain classes of trees (see e.g.\ Table 1 of \cite{gallianlabellingsurvey}), there seems to be much more difficulty in proving broader statements.
The best known general result about \cref{conj:gtc} due to Adamaszek, Allen, Grosu and Hladk\'{y} considers a variant of graceful labellings where a small number of additional labels are allowed. Using terminology of Van Bussel \cite{VanBussel2002}, given $m>n$ and an $n$-vertex tree $T$, an injective mapping $\phi:V(T)\rightarrow \{1,\ldots, m\}$ for which the edge-differences of $\phi$ are pairwise distinct is called a \textit{range-relaxed graceful labelling} of $T$. In 2020, they proved the following notable result.
\begin{theorem}[Adamaszek, Allen, Grosu and Hladk\'{y} \cite{adamaszek2020almost}] \label{thm:AAGH}
     For all $\eps> 0$ there exist $\eta, n_0 > 0$ such that for all $n>n_0$, the following holds.
     If $T$ is an $n$-vertex tree and $\Delta(T) \leq \frac{\eta n}{\log n}$, then there exists a range-relaxed graceful labelling $\phi : V(T) \rightarrow \{1,\ldots,(1 + \varepsilon)n\}$.
\end{theorem}

Alternatively, one could relax the notion of graceful labellings by allowing a small number of edge-differences to coincide, and we follow this approach. Introduced by
Erd\H{o}s, Hell and Winkler \cite{ERDOS1988117}, and denoted by \defn{$\gs(T)$}, the \defn{gracesize} of an $n$-vertex tree $T$ is the maximum size of the set of edge-differences over all bijective labellings $\phi : V(T) \rightarrow \{1,\ldots,n\}$. 
This was mentioned as a dual notion to the bandsize of a tree, 
which instead considers the minimum size of this set. Gracesize naturally relates to the graceful tree conjecture since such a labelling $\phi$ of $T$ is graceful if and only if the set of its edge-differences has size $|E(T)| = n-1$. Therefore, one can restate the graceful tree conjecture as: every $n$-vertex tree has gracesize $n-1$. 
It was suggested in \cite{ERDOS1988117} that it may be an interesting problem to find non-trivial lower bounds on the gracesize. 

In 1995, R\'osa and Širáň \cite{rosa_siran} proved that $\gs(T) \geq 5(n-1)/7$ for every $n$-vertex tree $T$. Without too much difficulty one can show that $\gs(T) = (1-o(1))n$ for all trees with maximum degree $o(n/\log n)$ as a corollary of \cref{thm:AAGH}\footnote{ 
One can argue similarly to our deduction of \cref{thm:gracesize} from \cref{lem:mainGL}.}. The main result of the paper confirms this bound on the gracesize of all sufficiently large trees, irrespective of maximum degree.
\begin{theorem}\label{thm:gracesize}
    For all $\eps> 0$ there exists $n_0$ such that for all $n>n_0$, the following holds. If $T$ is an $n$-vertex tree then 
    $\gs(T)\geq (1-\varepsilon)n$.
\end{theorem}

\subsection{Reduction argument}
Before reformulating the problem to be about edge-coloured graphs, we first show that in order to prove \cref{thm:gracesize}, it is enough to verify the following lemma, which is another approximately graceful result. Similar to the style of \cref{thm:AAGH}, we consider an extra relaxation on the number of labels used.
\begin{lemma}\label{lem:mainGL}
   For all $\eps> 0$ there exists $N$ such that for all $n>N$ the following holds.
    If $T$ is an $n$-vertex tree then there exists an injective mapping $\phi:V(T) \rightarrow \{1,\ldots,(1 + \varepsilon)n\}$ 
    such that the set of edge-differences of $\phi$ has size at least $(1-\eps)n$.
\end{lemma}

\begin{proof}[Proof of \cref{thm:gracesize} from \cref{lem:mainGL}]
    We can assume $\varepsilon<1$ as else the result is trivial, and let $n_0$ be sufficiently large so that for all $n>n_0$, the value of $(1-\eps/2)n$ is at least as large as the output $N$ of \cref{lem:mainGL} when applied with $\eps/2$. Let $T$ be an $n$-vertex tree. One-by-one, for $\varepsilon n / 2$ steps, remove a leaf from the current subtree of $T$ that remains, to obtain a subtree $T^* $ on $n^* \coloneqq (1-\varepsilon/2)n$ vertices. Apply \cref{lem:mainGL} to $T^*$ with $\varepsilon/2$ playing the role of $\eps$ to obtain an injective mapping $\phi^*:V(T^*)\rightarrow \{1,\ldots, (1+\varepsilon/2)n^*\}$ such that the set of edge-differences of $\phi^*$ has size at least $(1-\eps/2)n^*$. Since $(1+\varepsilon/2)n^* \leq n$ then from this we can define a bijective labelling $\phi:V(T) \rightarrow \{1,\ldots, n\}$ such that $\phi(x) = \phi^*(x)$ for all $x \in V(T^*)$ and the vertices in $V(T) \setminus V(T^*)$ are bijectively assigned a label from the set of unused labels in $[n]$. Since $E(T^*)\subseteq E(T)$ and $|\phi^*(x)-\phi^*(y)|=|\phi(x)-\phi(y)|$ for all $xy\in E(T^*)$, the set of edge-differences of $\phi $ is at least as large as the set of edge differences of $\phi^*$. Therefore $\gs(T)\geq (1-\eps/2)n^* \geq (1-\eps)n$, as desired.
\end{proof}

We remark that the relaxation given in \cref{thm:gracesize} is a slightly weaker notion than that of a range-relaxed graceful labelling in \cref{thm:AAGH}, in the sense that if a tree has a range-relaxed graceful labelling using $(1+o(1))|T|$ labels, then it has gracesize $(1-o(1))|T|$ (this can be seen via the reduction used in the proof that \cref{lem:mainGL} implies \cref{thm:gracesize}), while there is no obvious argument that gracesize $(1 - o(1))|T|$ implies the existence of a range-relaxed graceful labelling using $(1 + o(1))|T|$ labels. On the other hand, unlike \cref{thm:AAGH}, \cref{thm:gracesize} holds for all large trees, including those with high degree vertices.

Combining these two directions, a natural next step towards proving \cref{conj:gtc} would be to show that every tree $T$ admits a range-relaxed graceful labelling into $[(1+o(1))|T|]$, either via a reduction from gracesize, or by an independent argument.

\subsection{Rainbow reformulation}

Recall that we are interested in using edge-coloured graphs to study the graceful tree conjecture. To do this, we can define the difference coloured complete graph, as follows.

\begin{definition}\label{defn:differencecomplete}
 Let $X \subseteq \mathbb{N}$. The difference coloured complete graph for $X$, denoted by $K_{X}$, is the edge-coloured complete graph on vertex set $X$, where edge $ij$ is assigned colour $|i-j|$ for all distinct $i,j \in X$. 
\end{definition}

Consider the case when $X = [n] = \{1,\ldots,n\}$.
We can ask whether $K_{[n]}$ contains a rainbow copy of an $n$-vertex tree, meaning that every edge of the tree has a distinct colour. Given a tree $T$ and an injective function $\phi: V(T) \rightarrow [n]$, an embedding of $T$ in $K_{[n]}$ is uniquely determined from $\phi$ by embedding a given vertex $x$ to $\phi(x) \in V(K_{[n]})$. Note that an edge $xy \in E(T)$ has colour $c$ under this embedding if and only if $ c = |\phi(x) - \phi(y)|$. Similarly an embedding of $T$ uniquely determines such an injective function. 
In particular if $T$ has $n$ vertices, then observe that $\phi$ is graceful if and only if the corresponding embedding of $T$ is rainbow.
Therefore we can use $K_{[n]}$ to consider the following equivalent version of the graceful tree conjecture.

\begin{conjecture}[Rainbow version of the graceful tree conjecture]\label{conj:gracefulrainbow}
    $K_{[n]}$ contains a rainbow copy of every $n$-vertex tree $T$.
\end{conjecture}

\begin{figure}[b]
\centering
\begin{minipage}{.4\textwidth}
  \centering
    \begin{tikzpicture}[scale=0.9]
\draw[Goldenrod, line width=0.5mm] (0,1.5) -- (-1.495,0.455);
\draw[green, line width=0.5mm] (0,1.5) -- (1.495,0.455);
\draw[blue, line width=0.5mm] (0,1.5) -- (-0.88,-1.2135);
\draw[red, line width=0.5mm] (0,1.5) -- (0.88,-1.2135);
\draw[green, line width=0.5mm] (1.495,0.455) -- (0.88,-1.2135);
\draw[green, line width=0.5mm] (-1.495,0.455) -- (-0.88,-1.2135);
\draw[blue, line width=0.5mm] (-1.495,0.455) -- (1.495,0.455);
\draw[red, line width=0.5mm] (-1.495,0.455) -- (0.88,-1.2135);
\draw[red, line width=0.5mm] (1.495,0.455) -- (-0.88,-1.2135);
\draw[green, line width=0.5mm] (-0.88,-1.2135) -- (0.88,-1.2135);
\filldraw[black] (0,1.5) circle (2.6pt) node[anchor=south]{$1$};
\filldraw[black] (1.495,0.455) circle (2.6pt) node[anchor=west]{$2$};
\filldraw[black] (-1.495,0.455) circle (2.6pt) node[anchor=east]{$5$};
\filldraw[black] (0.88,-1.2135) circle (2.6pt) 
node[anchor=north]{$3$};
\filldraw[black] (-0.88,-1.2135) circle (2.6pt) node[anchor=north]{$4$};
\end{tikzpicture}
  \captionof{figure}{$K_{[5]}$}
  \label{fig:K5}
\end{minipage}%
\begin{minipage}{.4\textwidth}
  \centering
\begin{tikzpicture}[scale=0.9]
\draw[Goldenrod, line width=0.6mm] (0,1.5) -- (-1.495,0.455);
\draw[blue, line width=0.5mm] (-1.495,0.455) -- (1.495,0.455);
\draw[red, line width=0.5mm] (1.495,0.455) -- (-0.88,-1.2135);
\draw[green, line width=0.5mm] (-0.88,-1.2135) -- (0.88,-1.2135);
\filldraw[black] (0,1.5) circle (2.6pt) node[anchor=south]{$1$};
\filldraw[black] (1.495,0.455) circle (2.6pt) node[anchor=west]{$2$};
\filldraw[black] (-1.495,0.455) circle (2.6pt) node[anchor=east]{$5$};
\filldraw[black] (0.88,-1.2135) circle (2.6pt) 
node[anchor=north]{$3$};
\filldraw[black] (-0.88,-1.2135) circle (2.6pt) node[anchor=north]{$4$};
\end{tikzpicture}
  \captionof{figure}{Rainbow copy of $P_5$}
  \label{fig:P5}
\end{minipage}
\end{figure}

In this context, our main result (\cref{thm:gracesize}) corresponds to the statement that every $n$-vertex tree $T$ can be embedded into  $K_{[n]}$ such that at least $(1-o(1))n$ distinct colours appear on its edges. Since we have already shown that \cref{thm:gracesize} holds assuming the correctness of \cref{lem:mainGL}, proving this lemma will be the main focus for the remainder of this paper. We wish to attack it from a coloured perspective, so, rather than proving \cref{lem:mainGL} directly, we will actually show that the following equivalent version holds.

\begin{restatable}{lemma}{main}
\label{maintheorem}
    For all $\eps> 0$ there exists $n_0$ such that for all $n>n_0$ the following holds.
    If $T$ is an $n$-vertex tree, then 
    there exists an embedding of $T$ in $K_{[(1+\eps)n]}$ where at least $(1-\eps)n$ distinct colours appear on edges of $T$.
\end{restatable}

\section{Preliminaries}\label{sec:preliminaries}

\subsection{Organisation of the paper}
In the remainder of this section, we introduce notation and some standard probabilistic tools. \Cref{strategy} will give an overview of the proof strategy for \cref{maintheorem}, discussing the main methods needed to (i) find a nice structure in a tree $T$, and (ii) use the properties of this structure to embed parts of $T$ in a rainbow way into the difference coloured complete graph. These two points will be addressed in \cref{sec:structure,sec:TnotW} respectively. In \cref{subsection:PROOF}, we combine everything to prove \cref{maintheorem}.

\subsection{Notation}
For all $n \in \mathbb{N}$, let $[n] \coloneqq \{1,\ldots,n\}$. For any $a,b,c \in \mathbb{R}$ we write $a=b\pm c$ to mean $b-c\leq a \leq b+c$. 
If we say that a statement holds whenever $0 < a \ll b \leq1$, then there exists a non-decreasing function $f:(0,1] \rightarrow (0,1] $ such that the statement holds for all $0<a,b\leq1$ with $a\leq f(b)$. We similarly consider a hierarchy of constants $0 < b_1 \ll b_2 \ll \ldots \ll b_k <1 $, and these constants must be chosen from right to left. For every constant $b$ in a hierarchy it will be implicitly assumed that both $b\in (0,1)$ and $1/b \in \mathbb{N}$ are satisfied.

If $G$ is a graph we denote its vertex set and edge set by $V(G)$ and $E(G)$ respectively, and use the notation $|G| \coloneqq |V(G)|$ and $e(G) \coloneqq |E(G)|$. 
The degree and neighbourhood of a vertex $v \in V(G)$ are denoted by $d_G(v)$ and $N_G(v)$ respectively, and the neighbourhood of a set $U\subseteq V(G)$ is $N_G(U) = \bigcup_{u \in U}N_G(u)$. We may omit the subscripts where context is clear. The minimum degree of $G$ is $\delta(G)$ and the maximum degree is $\Delta(G)$. Given a set $U \subseteq V(G)$, we write $G[U]$ for the induced subgraph of $G$ on $U$. We use $G \setminus U$ to denote the subgraph $G[V(G)\setminus U]$, and for a single vertex $v$, we write $G - v$ instead of $G\setminus \{v\}$.  For an edge subset $M \subseteq E(G)$, $V(M)$ denotes the set of vertices which are contained in some edge belonging to $M$, and we say these vertices are covered by $M$. For any edge $xy \in E(G)$, $G-xy$ is the spanning subgraph obtained by deleting the edge $xy$ from $G$. 
If a graph $G$ is edge-coloured, then we denote by $C(G)$ the set of colours used on some edge in $G$. For a set $S$ of colours, we say that $G$ is $S$-rainbow if all edges in $G$ have a distinct colour 
from the set $S$. 

For graphs $F$ and $G$, we write $F + G$ to denote the disjoint union of $F$ and $G$. The sum is defined analogously for more than two graphs. For $d\in \mathbb{N}$, we write $F \times d$ to mean the graph obtained by taking the disjoint union of $d$ copies of $F$, that is, $F \times d \cong F + (F+ (F+ \ldots )) $ where the sum is taken $d-1$ times (so there are $d$ occurrences of $F$).

For a hypergraph $H$ we use the analogous notation $V(H), E(H), |H|, e(H), \delta(H), \Delta(H)$ as for graphs, whereby degree will always refer to vertex degree.
$H$ is $r$-uniform if each of its edges has size $r$, and it is $k$-partite if $V(H)$ can be partitioned into $k$ parts $U_1,\ldots,U_k$, such that $|e\cap U_i| \leq 1$ for every $e \in E(H),\ i\ \in [k]$.
A hypergraph $H$ is said to be linear if any pair of distinct edges in $H$ share at most one common vertex. 

\subsection{Probabilistic tools}

We will need some common probabilistic tools to show the existence of specified properties in edge-coloured graphs and hypergraphs. We write $X\sim \text{Bin}(n,p)$ to mean $X$ follows the Binomial distribution with parameters $n$ and $p$, and write $\mathbb{P}[\cdot]$ and $\mathbb{E}[\cdot]$ to denote the probability and expectation of a random variable. Given a set $N$ and $p\in [0,1]$, a $p$\textit{-random subset} $S\subseteq N$ is one that is formed by independently keeping every element in $N$ with probability $p$. Similarly if $p\leq 1/k$, then $S_1\ldots, S_k$ are \textit{pairwise disjoint} $p$\textit{-random subsets} of $N$ if they are formed by independently placing each element of $N$ in at most one $S_i$ so that the element is placed in $S_1$ with probability $p$, in $S_2$ with probability $p$, and so on, and placed in none of the sets with probability $1-kp$. These sets form a $p$\textit{-random partition} of $N$ if $p = 1/k$.

\begin{theorem}[Chernoff bound, see e.g.\ 
\cite{JansonSvante2000Rg/S}]\label{thm:chern}
Let $X\sim \text{\upshape{Bin}}(n,p)$.
For all $\lambda \in (0,1)$, 
$$\mathbb{P}[|X-\mathbb{E} [X]|\geq \lambda \mathbb{E} [X] ]\leq 2e^{-\frac{\lambda^2 \mathbb{E} X}{3}}.$$
\end{theorem}

 Let $Z_1,\ldots,Z_n$ be independent random variables, each taking values in a set $\Omega$. For a constant $c$, 
 we say that a function $f :\Omega^n \rightarrow \mathbb{R}$ is \textit{$c$-Lipschitz} if for any $\omega = (\omega_1,\ldots,\omega_n) \in \Omega^n$, changing the outcome of at most one $\omega_i$ affects the value of $f(\omega)$ by at most $c$. 

\begin{theorem}[McDiarmid's inequality \cite{mcdiarmid1989method}]\label{thm:McDiarmid}
   Let $Z_1,\ldots,Z_n$ be independent random variables, each taking values in a set $\Omega$. Let $f : \Omega^n \rightarrow \mathbb{R}$ be a $c$-Lipschitz function and consider the random variable $X = f(Z_1,\ldots,Z_n)$. For all $t > 0$,
    $$
    \mathbb{P}[|X-\mathbb{E} [X]|\geq t]\leq 2e^{-\frac{2t^2}{c^2n}}.
    $$
\end{theorem}
\section[Strategy]{Proof strategy for \Cref{maintheorem}}\label{strategy}
We have already observed that a graceful labelling of a tree $T$ with labels in $[n]$ can be thought of interchangeably as a rainbow embedding of $T$ in $K_{[n]}$. Therefore throughout, we will regard labellings and (edge-coloured) embeddings of graphs as equivalent, where the colour of an edge is uniquely determined by taking the absolute difference of the labels assigned to the endvertices of the edge. More explicitly, given two adjacent vertices $x,y \in V(T)$ for some tree $T$ together with a mapping $\phi:V(T) \rightarrow [n]$, we will refer to the colour of the edge $xy$ to mean the value of $|\phi(x)-\phi(y)|$.

Let us now give an overview of the main ideas used in our strategy to embed a given $n$-vertex tree $T$ into $K_{[(1+\varepsilon)n]}$ in an almost rainbow way. First, in \Cref{sec:structure}, we find a helpful structure within $T$ which is easier to work with. We split up $T$ by first isolating a set $S_{\text{high}} \subseteq V(T)$, consisting of all vertices of high degree. There cannot be too many high degree vertices, since $e(T) = n-1$, so we have an upper bound for $|S_{\text{high}}|$. We find a structural argument \cref{lem:structure4} which roughly states the following: for any subset of vertices $S\subseteq V(T)$ such that $|S|$ is not too large, there exists a set of vertices $W \subseteq V(T) \setminus S$, also of small order, such that deleting $S \cup W$ from $T$ yields many vertex-disjoint copies of some small forest $F$. Applying this with $S_{\text{high}}$ playing the role of $S$, we can think of the set $W$ as a set of `waste' vertices in $T$, of which the lemma tells us there are not too many. So, we have that $T \setminus (S_{\text{high}} \cup W)$ is exactly made up of many copies of some small forest $F$.

Our new aim is to find a fully rainbow copy of $T\setminus W$ in $K_{[(1+\varepsilon)n]}$, since at the end, the set $W$ will be embedded arbitrarily to the available leftover labels, and this is where we may obtain colour repetitions. There will not be too many repeats since all vertices in $W$ will have low degree, so that the number of edges with an endvertex in $W$ is small. The embedding of $T\setminus W$ is the main focus in \Cref{sec:TnotW}, with the key lemma given as \cref{lem:TminusJrooted_comps}. So, let us now summarise this process, noting that $T \setminus W$ is comprised of the set $S_{\text{high}}$ and many vertex-disjoint copies of $F$ (see \cref{fig:T_W_2}). 
We define an auxiliary tree $T_{\text{aux}}$ from $T\setminus W$ obtained by contracting the set $S_{\text{high}}$ into a single vertex $v$, and contracting sets of vertices in the copies of $F$ (chosen so that these sets of vertices are copies of the same vertex of $F$) so that the resulting tree $T_{\text{aux}}$ consists of one high degree vertex $v$ adjacent to constantly many copies of $F$ (see \cref{fig:T_aux_2}).

\begin{figure}
\centering
\begin{minipage}{.4\textwidth}
  \centering
            \begin{tikzpicture}[scale=1]
            \filldraw[ProcessBlue!30!white] (-2,0.25) ellipse (0.5cm and 1.2cm) node[anchor=east, xshift = -1.5em]{\color{black}$S_{\text{high}}$}; 
            \filldraw[Goldenrod!50!white] (-0.5, 1.25) ellipse (0.4cm and 0.9cm); 
            \filldraw[RedOrange!40!white] (-0.5, -1) ellipse (0.4cm and 0.9cm); 
            \filldraw[Lavender!50!white] (1.5,2) circle (6pt); 
            \filldraw[Blue!30!white] (1.5,1.5) circle (6pt); 
            \filldraw[Lavender!50!white] (1.5,1) circle (6pt); 
            \filldraw[Blue!30!white] (1.5,0.5) circle (6pt); 
            \filldraw[Black!30!white] (1.5,-1.75) circle (6pt); 
            \filldraw[red!50!white] (1.5,-1.25) circle (6pt); 
            \filldraw[Black!30!white] (1.5,-0.75) circle (6pt); 
            \filldraw[red!50!white] (1.5,-0.25) circle (6pt); 
            
\draw[black,  line width = 0.6mm] (-2,1) -- (-0.5,1.75);
\draw[black,  line width = 0.6mm] (-2,0.5) -- (-0.5,0.75);
\draw[black,  line width = 0.6mm] (-2,1) -- (-2,0.5);
\draw[black,  line width = 0.6mm] (-2,-0.5) -- (-.5,-0.5);
\draw[black,  line width = 0.6mm] (-2,-0.5) -- (-.5,-1.5);
\draw[black,  line width = 0.6mm] (-.5,-1.5) -- (1.5,-1.75);
\draw[black,  line width = 0.6mm] (-.5,-1.5) -- (1.5,-1.25);
\draw[black,  line width = 0.6mm] (-.5,-0.5) -- (1.5,-0.25);
\draw[black,  line width = 0.6mm] (-.5,-0.5) -- (1.5,-0.75);
\draw[black,  line width = 0.6mm] (-.5,1.75) -- (1.5,2);
\draw[black,  line width = 0.6mm] (-.5,1.75) -- (1.5,1.5);
\draw[black,  line width = 0.6mm] (-.5,0.75) -- (1.5,1);
\draw[black, line width = 0.6mm] (-.5,0.75) -- (1.5,0.5);
\filldraw[black] (-2,1) circle (2.6pt); 
\filldraw[black] (-2,0.5) circle (2.6pt);
\filldraw[black] (-2,-0.5) circle (2.6pt) ;
\filldraw[black] (-0.5,0.75) circle (2.6pt); 
\filldraw[black] (-0.5,1.75) circle (2.6pt);
\filldraw[black] (-0.5,-0.5) circle (2.6pt) ;
\filldraw[black] (-0.5,-1.5) circle (2.6pt) ;
\filldraw[black] (1.5,2) circle (2.6pt); 
\filldraw[black] (1.5,1.5) circle (2.6pt);
\filldraw[black] (1.5,1) circle (2.6pt) ;
\filldraw[black] (1.5,0.5) circle (2.6pt) ;
\filldraw[black] (1.5,-1.75) circle (2.6pt); 
\filldraw[black] (1.5,-1.25) circle (2.6pt);
\filldraw[black] (1.5,-0.25) circle (2.6pt) ;
\filldraw[black] (1.5,-0.75) circle (2.6pt) ;
\node[] at (0.5,0.15) {\textbf{\ldots}};
\filldraw[black] (-2,0.05) circle (0.7pt) ;
\filldraw[black] (-2,0.2) circle (0.7pt) ;
\filldraw[black] (-2,-0.1) circle (0.7pt) ;
\draw [decorate, decoration = {brace, raise = 2pt, amplitude = 14pt}] (1.8,2) --  (1.8,-1.95) node[midway,xshift=4em,yshift=0.5em]{\small Copies of }
node[midway,xshift=4em, yshift =-0.65em]{\small forest $F$};
            \end{tikzpicture}
  \captionof{figure}{Example of $T\setminus W$ where $F$ is the star with two leaves, rooted at the centre.}
  \label{fig:T_W_2}
\end{minipage}%
\hspace{12mm}
\begin{minipage}{.4\textwidth}
  \centering
 \begin{tikzpicture}[scale=1]
        \filldraw[ProcessBlue!30!white] (-2,0.5) circle (8pt) node[anchor=east, xshift = -0.8em]{\color{Black} $v$}; 

            \filldraw[Goldenrod!50!white] (0,1.5) circle (8pt); 
            \filldraw[RedOrange!40!white] (0,-0.5) circle (8pt); 
            \filldraw[Lavender!50!white] (2,2) circle (8pt); 
            \filldraw[Blue!30!white] (2,1) circle (8pt); 
            \filldraw[Black!30!white] (2,-1) circle (8pt); 
            \filldraw[red!40!white] (2,0) circle (8pt); 
 \draw[black, line width=0.6mm] (0,1.5) -- (-2,0.5);
\draw[black, line width=0.6mm] (-2,0.5) -- (0,-0.5);
\draw[black, line width=0.6mm] (0,1.5) -- (2,2);
\draw[black, line width=0.6mm] (0,1.5) -- (2,1);
\draw[black, line width=0.6mm] (0,-0.5) -- (2,-1);
\draw[black, line width=0.6mm] (0,-0.5) -- (2,0);
\filldraw[black] (-2,0.5) circle (2.6pt) ;
\filldraw[black] (0,1.5) circle (2.6pt) ;
\filldraw[black] (2,1) circle (2.6pt);
\filldraw[black] (2,2) circle (2.6pt);
\filldraw[black] (2,-0) circle (2.6pt) ;
\filldraw[black] (2,-1) circle (2.6pt) ;
\filldraw[black] (0,-0.5) circle (2.6pt);
\node[] at (1,0.4) {\textbf{\ldots}};
\draw [decorate, decoration = {brace, raise = 2pt, amplitude = 14pt}] (2.2,2.3) --  (2.2,-1.3) node[midway,xshift=4.5em,yshift=0.5em]{\small Fewer copies}
node[midway,xshift=4.5em, yshift =-0.65em]{\small of $F$};
            \end{tikzpicture}
  \captionof{figure}{Auxiliary tree $\taux$. Each vertex has a colour, and corresponds to the contraction of the vertices of the same colour in \cref{fig:T_W_2}.}
  \label{fig:T_aux_2}
\end{minipage}
\end{figure}

We observe that for any integer $d$, a copy of $F \times d$, that is, the union of $d$ vertex-disjoint copies of $F$, can be constructed by taking the union of $e(F)$ pairwise edge-disjoint matchings, each of size $d$, and each of which corresponds to a distinct edge of $F$. Thus we find tools for embedding rainbow matchings in \Cref{lem:mathcingSiSjv3}, and use these to find a rainbow embedding of $T_{\text{aux}}$ inside $K_{[(1+\varepsilon/2)|T_{\text{aux}}|]}$. We consider a rainbow `blow up' of this auxiliary tree in $K_{[(1+\varepsilon)n]}$ by thinking of each vertex as a set (these will correspond to the contracted vertex sets) and by replacing each edge with some rainbow matching.  
With the flexibility from the extra labels, this allows us to manipulate the properties of the blow up and be more careful with which sections we are allowed to embed into, in order to obtain a rainbow embedding of $T \setminus W$ in $K_{[(1+\varepsilon)n]}$, as desired. 

For this argument to work we require that both $|S_{\text{high}}|$ is small, and all vertices in $W$ have low degree. The former condition holds using the fact that all vertices in $S_{\text{high}}$ have degree bigger than some constant $\Delta$, so since $e(T) = n-1$ there must be fewer than $2n/\Delta$ of them. But if we choose $\Delta$ to be too high, then we cannot get a good upper bound for the maximum degree of the waste vertices. Therefore, for these conditions to be compatible, we need to be particularly careful with how we select the parameters which define these sets. More detail of this is given prior to the full proof, in \cref{subsection:PROOF}. 

\section{Tree splitting}\label{sec:structure}
As mentioned in the previous section, in order to prove \cref{maintheorem}, we start by finding a way to delete certain vertices in a tree to yield many copies of some forest $F$, satisfying various additional properties. In this section, our main aim is to prove the following lemma.
\begin{lemma}\label{lem:structure4}
For all $\delta> 0$ there exist $\zeta_0,n_0>0$ such that the following holds for all $n>n_0$ and $\zeta <\zeta_0$ satisfying $\zeta n \in \N$. 
    Let $T$ be an $n$-vertex tree and let $S \subseteq V(T)$ have order at most $\delta n/10$. There exists a set $W\subseteq V(T) \setminus S$ of order at most $\delta n$ such that $T \setminus W$ satisfies the following properties:
    \begin{enumerate}[label = \upshape{(\roman*)}]
        \item $T \setminus (W\cup S) \cong F \times \zeta n$ for some forest $F$ which has rooted trees as components; and
        \item for every component of $F \times \zeta n$, the root has at most one neighbour in $S$, and all other vertices in the component have no neighbour in $S$.
    \end{enumerate}
\end{lemma}
We need to introduce some further results that demonstrate various properties about trees and forests, which we will then combine to prove \cref{lem:structure4} at the end of this section.

\begin{lemma}\label{lem:structureoneedgeJtree}
 Let $T$ be an $n$-vertex tree. For every $S\subseteq V(T)$, 
 there is a set $W \subseteq V(T) \setminus S$ satisfying 
 $|W|\leq 2|S|$, such that for every component $C$ of $T \setminus (S\cup W)$, there is at most one edge between $C$ and $S$.
\end{lemma}

\begin{proof}
    Let $T_S$ be the smallest subtree of $T$ which contains all vertices in $S$, noting that this is unique. We first prove by induction on $|S|$ that $|N_{T_S}(S)\setminus S|\leq 2|S|$. 
    If $|S| = 1$, then $T_S$ consists only of the single vertex in $S$ and $N_{T_S}(S) = \emptyset$, so there is nothing to prove. Assume now that $|S|\geq 2$ and for all $S'\subseteq V(T)$ with $|S'|<|S|$, we have $|N_{T_{S'}}(S')\setminus S'|\leq 2|S'|$.

    Since we chose $T_S$ minimally, all leaves of $T_S$ must belong in $S$. Otherwise, we could delete a leaf outside of this set and find a smaller subtree still containing $S$. Since $|T_S|\geq |S|\geq 2$ there exists a leaf of $T_S$, denote one by $v$. Let $S' = S\setminus \{v\}$ so that $|S'| = |S|-1$. We have $T_{S'} \subset T_{S}$ and there is a path $P$ in $T_S$ that connects $v$ to $V(T_{S'})$ and has no internal vertices in $V(T_{S'})$, so that $T_S = T_{S'}\cup P$. Let $w$ denote the endvertex of P that belongs to $V(T_{S'})$. Passing from $S'$ to $S$ can only introduce at most two new vertices into $N_{T_S}\setminus S$, namely the neighbour of $v$ on $P$ and the neighbour of $w$ on $P$, therefore by assumption, this set has size at most $|N_{T_{S'}}(S')\setminus S'|+2 \leq 2|S|$, as desired.

    Now, let $W = N_{T_S}(S)\setminus S$ so that $|W|\leq 2|S|$. Suppose there exists a component $C$ in $T\setminus (S\cup W)$ which sends two distinct edges into $S$, say $xs$ and $yt$ for some $x,y \in V(C)$, $s,t \in S$. Since $T_S$ and $C$ are both connected, then there is a path $P$ from $s$ to $t$ contained in $T_S$, and there is a path $Q$ from $x$ to $y$ contained in $C$. We know that at least one of these paths is non-empty as otherwise $x=y$, $s=t$, which implies $xs = yt$. If $V(P)\cap V(Q) =\emptyset$, then we find a cycle in $T$ consisting of the edges $xs$ and $yt$ and the paths $P$ and $Q$, a contradiction.
    Otherwise, choose $z \in V(P)\cap V(Q)$ to be of minimal distance from $x$ along the path $Q$. Then the subpath $Q'$ of $Q$ from $x$ to $z$ is edge-disjoint from the subpath $P'$ of $P$ from $z$ to $s$.
    Since $z \in V(Q) \subseteq C$, then $z\notin S \cup N_{T_S}(S)$, implying that $P'$ is a path of length at least $2$. Therefore, we find a cycle in $T$ of length at least three, consisting of $Q'$, $P'$ and the edge $xs$, again a contradiction. We conclude that every component of $T \setminus (S \cup W)$ touches at most one edge going into $S$.
\end{proof}

We use the following standard observation about trees, providing a proof for completeness.
\begin{fact}\label{lem:smallcompsbasecase}
    Every $n$-vertex tree $T$ contains a vertex $v \in V(T)$ such that all components of $T -v$ have order at most $n/2$. 
\end{fact}

\begin{proof}
    Let us consider an orientation of the tree $T$ given as follows: for every $xy\in E(T)$, direct the edge from $x$ to $y$ if the component of $T - xy$ containing $x$ is smaller than the one containing $y$ (and orient the edge $xy$ arbitrarily if these components have the same size). Every oriented tree contains a sink, that is, a vertex $v \in V(T)$ where $v$ has only in-neighbours. Any component $C$ in $T-v$ contains some neighbour $u$ of $v$ since $T$ is connected, and thus $C$ is smaller than the component of $T-uv$ which contains $v$. So, $|C|\leq n/2$.
    \end{proof}

This fact provides a means to prove the inductive step of the next straightforward lemma.
\begin{lemma}\label{lem:smallcomponents}
    Let $k\in \mathbb{N} \cup \{0\}$. For every $n$-vertex tree $T$ there exists a set $W\subseteq V(T)$ of order at most $ 1+ 2+\dots + 2^k$, such that all components of $T\setminus W$ have order at most $n/2^k$.
\end{lemma}

\begin{proof}
    We proceed by induction on $k$. The case $k=0$ holds trivially. Suppose the statement holds true for some $k$, and we want to prove it for $k+1$. Consider the set $W'$ obtained by applying the result for $k$, so that $|W'|\leq 1+ 2 + \ldots + 2^k$ and all components of $T \setminus W'$ have order at most $n/2^k$. Let $C_1,\ldots,C_t$ denote the set of components of $T\setminus W'$ which have order larger than $n/2^{k+1}$. Clearly $t \leq 2^{k+1}$ since these components are pairwise vertex-disjoint and $|T|=n$. For each $i \in [t]$, we apply \cref{lem:smallcompsbasecase} to $C_i$ to find a vertex $v_i \in V(C_i)$ such that all components of $C_i -v_i$ have order at most $|C_i|/2 \leq n/2^{k+1}$. Then $W \coloneqq W' \cup \{v_1,\ldots,v_t\}$ has order  $|W'|+t \leq 1+ 2+\dots + 2^k + 2^{k+1}$ and all components of $T \setminus W$ have order at most $n/2^{k+1}$, as required.
\end{proof}

\begin{lemma}\label{lem:structurecompsorderm}
    Let $m,n\in\mathbb{N}$. For every $n$-vertex tree $T$ there exists a set $W \subseteq V(T)$ satisfying $|W| \leq \frac{4n}{m}$ and such that every component of $T \setminus W$ has order at most $m$.
\end{lemma}

\begin{proof}
    Let $k \coloneqq \lceil \log_2(n/m)\rceil$. Applying \cref{lem:smallcomponents} to $T$ we have that there exists a set $W$ of order at most $1+2+ \dots + 2^{k} = 2^{k+1} -1 \leq 2^{\log_2(n/m) +2} = 4n/m$, such that all components of $T\setminus W$ have order at most $n/2^k \leq n/2^{\log_2(n/m)} = m$.
\end{proof}

Finally we find a way to delete vertices from a forest $H$ so that the subgraph obtained consists precisely of disjoint copies of a new forest $F$ with a specified rooted structure.

\begin{lemma}\label{lem:deletevxscopyF}
Let $m,n\in\mathbb{N}$, $\zeta >0$ be such that $\zeta n \in \N$ and suppose $H$ is a forest with components of order at most $m$ and in each component we have specified a root. We can delete at most $ m^{m+1}\zeta n$ vertices from $H$ to get a copy of  $ F \times \zeta n$ for some forest $F$ which has rooted trees as components, and such that all roots in $F\times \zeta n$ were also roots in $H$. 
\end{lemma}

\begin{proof}
    By Cayley's formula there exist at most $m^{m-1}$ trees of order at most $m$, and for each of them there are at most $m$ ways to select a root. So there are at most $m^{m}$ types of rooted trees amongst the components of $H$. For each such rooted tree $T$, we delete at most $\zeta n$ copies of $T$ so that the remaining number of copies is divisible by $\zeta n$. The remaining graph is a copy of $F\times \zeta n$ for some forest $F$ where all components are rooted trees. In total we have deleted at most $m^{m} \times m \times \zeta n = m^{m+1} \zeta n $ vertices to obtain this.
\end{proof}

We are now ready to prove \cref{lem:structure4}, combining the properties we have seen already.
\begin{proof}[Proof of \cref{lem:structure4}]
   Let $\delta>0$, $m \coloneqq \frac{10}{\delta}$, $\zeta_0 \coloneqq \frac{2\delta}{5m^{m+1}}$ and let $n_0$ be sufficiently large.
   Let $n>n_0$ and $\zeta<\zeta_0$ be such that $\zeta n\in \N$, and let $T$ and $S$ be as in the statement of the lemma. We construct $W$ by deleting a bounded number of vertices. 
   By \cref{lem:structureoneedgeJtree}, there exists a set $W_1 \subseteq V(T) \setminus S$ of order at most $2|S|$ such that every component of $T \setminus (S \cup W_1)$ touches at most one edge that is incident to $S$. By \cref{lem:structurecompsorderm}, there exists a set $W_2\subseteq V(T)$ of order at most $4n/m$ such that all components in $T\setminus W_2$ have order at most $m$. So, each component in $H \coloneqq T\setminus (S\cup W_1 \cup W_2)$ has at most one vertex belonging in $N_T(S)$, and if such a vertex exists then call this the root, otherwise arbitrarily choose a root within the component. Then we can apply \cref{lem:deletevxscopyF} to $H$ to obtain a set $W_3 \subseteq V(H)$ of order at most $m^{m+1} \zeta n$ such that $H \setminus W_3 = T \setminus (S \cup W_1 \cup W_2 \cup W_3) \cong F \times \zeta n$ for some forest $F$ with rooted trees as components, where only the roots belong in $N_T(S)$. Clearly every component of $H\setminus W_3$ is a subgraph of a component of $H$, so using the property of $W_1$, we also know that every root has at most one neighbour in $S$. Take $W \coloneqq ( W_1 \cup W_2 \cup W_3)\setminus S$ to get 
   $|W| \leq 2|S|+\frac{4n}{m} + m^{m+1}\zeta n \leq \left(\frac{2\delta}{10} +\frac{4\delta}{10} +\frac{4\delta}{10} \right) n = \delta n$. Therefore $T \setminus W$ satisfies the desired conditions.
\end{proof}

\section{Finding rainbow subgraphs}\label{sec:TnotW}

\subsection{Using hypergraphs for matchings}
Observe that given a forest $F$ and some integer $d$, a rainbow copy of $F \times d$ is exactly the union of $e(F)$ colour-disjoint rainbow matchings, where each matching has size $d$ and corresponds to a unique edge in $F$. By using \cref{lem:structure4}, if we can find rainbow matchings in $K_{[n]}$, this will give us tools for embedding parts of trees in a rainbow way.
It will be helpful to consider $3$-uniform $3$-partite hypergraphs by making the following observation. Consider two disjoint vertex sets $A,B \subseteq [n]$ and a colour set $C \subseteq C(K_{[n]})$, let $K_{[n]}[A,B,C]$ denote the edge-coloured subgraph of $K_{[n]}$ containing only edges between $A$ and $B$ which have colour in $C$. We can consider the corresponding $3$-uniform $3$-partite hypergraph, denoted by $H_{[n]}[A,B,C]$, with vertex classes $A, B $ and $ C$ such that for all $a \in A,\   b\in B,\  c\in C$, the size-three edge $abc$ is an edge in $H_{[n]}[A,B,C]$ if and only if the edge $ab \in E(K_{[n]})$ has colour $c$, or equivalently $c = |b-a|$. 
In particular, there is a mapping $\theta : E(K_{[n]}[A,B,C]) \rightarrow E(H_{[n]}[A,B,C])$ which maps an edge $ab \in E(K_{[n]}[A,B,C])$ to the size-three edge $\{a,b,|b-a|\}\in E(H_{[n]}[A,B,C])$ where $|b-a| \in C$. Thus we observe the following.

\begin{observation}\label{obs:Matchings}
    A subset $M \subseteq E(K_{[n]}[A,B,C])$ is a rainbow matching if and only if $\theta(M)$ is a hypergraph matching in $H_{[n]}[A,B,C]$.
\end{observation}

In order to use this observation, we will first need some preliminary lemmas about $K_{[n]}[A,B,C]$.

\begin{lemma}\label{lem:spanninggraph}
    Let $s,n\in \mathbb{N}$ be such that $s\leq n/4$ and $n$ is even. Let $A \coloneqq [1,\frac{n}{2}]$, $B \coloneqq [\frac{n}{2} + 1,{n}]$ and $C\coloneqq [n-1]$. There exists a spanning subgraph $G \subseteq K_{[n]}[A,B,C]$ such that $\Delta(G)=2s$, every vertex in $[2s,\frac{n}{2}-s] \cup [\frac{n}{2}+2s,n-s]$ has degree $2s$, every colour in $[2s,n-2s]$ appears on exactly $s$ edges, no colour in $C$ appears on more than $s$ edges, and no edge has colour $1$.
\end{lemma}

\begin{proof} Note that $C = C(K_{[n]})$ so $C$ contains all possible colours in $K_{[n]}$. We will find our desired spanning subgraph $G$ to be contained in $ K_{[n]}[A,B,C]$, meaning that all edges will lie between $A$ and $B$.

For every $c \in C$, let $E_c \coloneqq \left \{\left( \frac{n-c+i}{2} , \frac{n+c+i}{2} \right) : i = 1,2, \ldots, 2s \right \} \cap \mathbb{Z}^2$.
Then $|E_c| = s$ for all $ c \in C$. Indeed, both of the values $n-c$ and $n+c$ are odd or they are both even, implying that the terms $\frac{n-c+i}{2} $ and $\frac{n+c+i}{2}$ are integer valued exactly when $i$ is odd or even respectively. Furthermore any edge in $E_c \cap E(K_{[n]})$ has colour $c$, since for a fixed $i$, we have $\frac{n+c+i}{2} - \frac{n-c+i}{2} = \frac{2c}{2} = c$.

    Now, let $C' \coloneqq [2s,n-2s]$ and fix $c \in C'$. For any $i \in [2s]$, note that
    $$ \frac{n-c+i}{2} \geq \frac{n-(n-2s)+i}{2} \geq \frac{2s+1}{2}\geq 1,$$ 
    $$\frac{n-c+i}{2} \leq \frac{n-2s+i}{2} \leq \frac{n-2s+2s}{2} =\frac{n}{2},$$
    $$ \frac{n+c+i}{2} \geq \frac{n+2s+i}{2} \geq \frac{n+ 2s+1}{2}\geq \frac{n}{2}+1,$$ 
    $$\frac{n+c+i}{2} \leq \frac{n+(n-2s)+i}{2} \leq \frac{2n-2s+2s}{2} =n.$$
    
    This shows that $E_c \subseteq A \times B $ for all $c \in C'$, i.e.\ every pair in $E_c$ forms an edge in $K_{[n]}$ of colour $c$, going from $A$ to $B$.
    Let $G$ be the spanning subgraph of $K_{[n]}$ defined by its edge set $$E(G) = \bigcup_{c\in C \setminus \{1\}}E_c \cap (A \times B).$$ So $G$ can be constructed by first taking edges from all $E_c \cap E(K_{[n]})$ sets, except for when $c= 1$, and deleting edges not lying between $A$ and $B$. Thus no edge in $G$ has colour $1$. Since $n$ is even, we have $E_1 \cap (A\times B) = \{(n/2,n/2+1)\}$, so by not adding this edge set to $G$ we only affect the degrees of only two vertices, namely $n/2$ and $n/2+1$, neither of which belong in $[2s,\frac{n}{2}-s] \cup [\frac{n}{2}+2s,n-s]$. For any $c \in C'$, since $E_c \subseteq A \times B $, then this implies $E_c \subseteq E(G)$ and in particular every colour $ c \in C'$ occurs on exactly $ |E_c|= s$ edges of $G$, as desired. 
    Any colour in $C\setminus C'$ clearly appears on at most $s$ edges.
    It remains to check the degrees of vertices in $A$ and $B$.

    First consider $A' \coloneqq [2s, \frac{n}{2}-s]$, and fix $x \in A'$. Let $C(x) \coloneqq \{c\in C : \frac{n-c+i}{2} = x \text{ and } \frac{n+c+i}{2} \in B \text{ for some } i \in [2s]\}$. Note that
    \begin{equation*}
          \frac{n-c+i}{2} = x \iff c = n+i - 2x,  
    \end{equation*}
    so we deduce that any $c \in C(x)$ satisfies $c \leq n+2s -2x \leq n+2s -4s = n-2s$ and $c \geq n+1 - 2x \geq n+1 - (n-2s) = 1+2s$.
    This implies $C(x) \subseteq C'$ and $1\notin C(x)$. In particular $C(x)$ is the set of colours $c\in C$ for which there exists $y \in B$ such that $xy \in E(G)$ has colour $c$. If there exist two edges in $G$ of colour $c$ that contain $x$, then there exist distinct $y,y'\in B$ such that $xy$ and $xy'$ are edges of colour $c$. By choice of $A$ and $B$, we know $y\geq x$ and $y'\geq x$ and so $c=y-x = y' - x$, implying $y=y'$, a contradiction. Thus for each colour $c \in C(x)$, there is exactly one edge in $G$ of colour $c$ with $x$ as an endvertex. So in total the number of edges in $G$ containing $x$ is $|C(x)| = 2s$, and in particular we have $d_G(x) = 2s$ for every $x\in A'$.

    Similarly, let $B' = [\frac{n}{2}+2s, n-s]$, and fix $y\in B'$. Let $D(y) \coloneqq \{c\in C : \frac{n+c+i}{2} = y \text{ and } \frac{n-c+i}{2} \in A \text{ for some } i \in [2s]\}$.
    Note that
    $$
    \frac{n+c+i}{2} = y \iff c = 2y - (n+i),
    $$
    so we deduce that any $c \in D(y)$ satisfies $c \leq 2y - (n+1)  \leq 2n-2s -n -1 = n-2s -1$ and $c \geq 2y - (n+2s) \geq n+ 4s - n - 2s = 2s$. Thus $D(y) \subseteq C'$ and $1\notin D(y)$. We may argue  similarly that there exists exactly one edge in $G$ of each colour in $D(y)$ which contains $y$. So there are $|D(y)| =2s$ edges in $G$ containing $y$ and we deduce that $d_G(y) = 2s$ for all $y \in B'$.
    Finally all vertices in $(A\setminus A') \cup (B\setminus B')$ are contained in at most $2s$ edges by choice of the $E_c$ sets, so $\Delta(G) \leq 2s$. Thus the graph $G$ satisfies our desired properties.
\end{proof}

Consider again the subgraph $K_{[n]}[A,B,C]$ of $K_{[n]}$ for disjoint vertex sets $A,B\subseteq [n]$ and a colour set $C \subseteq C(K_{[n]})$.
If every element of $A$ is smaller than every element of $B$, then observe that the corresponding hypergraph $H_{[n]}[A,B,C]$ is linear. We would like to find matchings in this hypergraph, in accordance with \cref{obs:Matchings}. We will use the following variant of an unpublished result from Pippenger (see e.g.\ \cite{alon_hypergaphs}) in order to find almost perfect matchings in linear uniform hypergraphs. 

\begin{theorem}\label{thm:Pippenger}
    Let $n^{-1}\ll \gamma \ll \alpha,\mu, r^{-1}$.
        Every $n$-vertex linear $r$-uniform hypergraph $H$ with maximum degree at most $(1+\gamma)\alpha n$ and at most $\mu n$ vertices of degree less than $(1-\gamma)\alpha n$ has a matching covering all but at most $2\mu n$ vertices.
\end{theorem}

\begin{proof}
     Let $H$ be as in the statement of the lemma. A result of Molloy and Reed (see \cite[Theorem 1]{Molloy_listcolouring}), tells us that the chromatic index $\chi'(H)$ is at most $\Delta + c\Delta^{1-1/r}(\log\Delta)^4$ for some constant $c$, where $\Delta $ is the maximum degree of $H$. Thus there exists a proper edge-colouring of $H$ using at most $ \chi'(H)\leq (1+\gamma)\alpha n + ((1+\gamma)\alpha n)^{1-1/(r+1)}$ colours, by choosing $n$ sufficiently large. Each colour class forms a matching. Since at least $(1-\mu)n$ vertices in $H$ have degree at least $(1-\gamma)\alpha n$, then $e(H) \geq (1-\mu)(1-\gamma)\alpha n^2/r$.  
  Thus there is a colour class of size at least \begin{align*}
      \frac{e(H)}{\chi'(H)} \geq \frac{(1-\mu)(1-\gamma)\alpha n^2/r}{(1+\gamma)\alpha n + ((1+\gamma)\alpha n)^{1-\frac{1}{r+1}}} &= (1-\mu)\frac{n}{r}\left(\frac{1-\gamma}{1+\gamma + (1+\gamma)^{1-\frac{1}{r+1}}(\alpha n)^{-\frac{1}{r+1}}}\right)\\
      &\geq (1-\mu)\frac{n}{r}\left(\frac{1-\gamma}{1+3\gamma}\right)\\
      & \geq (1-\mu)^2\frac{n}{r}\\
      & \geq (1-2\mu)\frac{n}{r}.
  \end{align*} 
  The second line holds by assuming $n$ is sufficiently large such that $(1+\gamma)^{1-\frac{1}{r+1}}(\alpha n)^{-\frac{1}{r+1}} \leq 2\gamma$, and the third holds since $\frac{1-\gamma}{1+3\gamma} \geq 1- 4\gamma \geq 1 - \mu$.
Since this colour class contains at least $(1-2\mu)n/r$ pairwise disjoint edges, each of size $r$, then the number of vertices covered by this matching is at least $(1-2\mu)n$. Thus the matching covers all but at most $2\mu n$ vertices in $H$.
\end{proof}

\begin{lemma}\label{lem:mathcingSiSjv3} Let $n^{-1}\ll p,\mu <1$ be such that $n$ is even. 
    Suppose $S_1, S_2 \subseteq [n]$ are $p$-random and disjoint. Then with probability $1-o(1)$ there exists a rainbow matching in $K_{[n]}[S_1,S_2,S_2]$ such that no edge in the matching has colour $1$, and it covers all but at most $ 2\mu\max\{|S_1|,|S_2|\}$ vertices in $S_1$ and all but at most $2 \mu\max\{|S_1|,|S_2|\}$ vertices in $S_2$.
\end{lemma}

\begin{proof}  Choose $\gamma \in (0,1)$ such that $n^{-1}\ll \gamma \ll p,\mu$ and let $n,S_1,S_2$ be as in the statement of the lemma.  Let $A \coloneqq [1,\frac{n}{2}]$, $B \coloneqq [\frac{n}{2} + 1,{n}]$ and $C\coloneqq [n-1]$.
Let $m \coloneqq 3pn/2$ and $s\coloneqq \mu m/20$, noting that $s\leq n/4$. Let $G \subseteq K_{[n]}[A,B,C]$ denote the spanning subgraph obtained by applying \Cref{lem:spanninggraph} and take $R\subseteq H_{[n]}[A,B,C]$ to be the $3$-partite $3$-uniform hypergraph corresponding to $G$, that is, there is an edge $abc \in E(R)$ if and only if $c=|b-a|$ and $ab \in E(G)$ for some $a\in A,\ b\in B$ and $c\in C$. We will consider two vertex-disjoint colour-disjoint random subhypergraphs of $R$ and find large matchings within both of these. We will then combine them to form an almost perfect matching $M$ in $H_{[n]}[S_1,S_2,S_2]$. By \Cref{obs:Matchings}, this then corresponds to an almost perfect rainbow matching $\theta^{-1}(M)$ in $K_{[n]}[S_1,S_2,S_2]$. 
    For that purpose, let us define various random subsets of vertices. 
    
    For $i\in\{1,2\}$, let $A_{i} := A \cap S_{i}$ and $B_i := B \cap S_{i}$. Partition $C$ into two $\frac{1}{2}$-random subsets $P$ and $Q$, and define $P_2 \coloneqq S_2 \cap P$ and $Q_2 \coloneqq S_2 \cap Q$.
    Note that, for any $a \in A$, $ b \in B$ and for any $i,j\in [2]$, we have $\mathbb{P}[a \in A_i] =\mathbb{P}[b \in B_j] = p $ and the corresponding events are independent. Also, for any $c \in C$, we have $\mathbb{P}[c\in P_2] = \mathbb{P}[c\in Q_2] = \frac{p}{2}$.
    We will consider two $3$-partite $3$-uniform subhypergraphs of $R$ induced on these subsets, defined by $H_P \coloneqq R[A_1,B_2, P_2]$ and $H_Q \coloneqq R[A_2,B_1, Q_2]$.
    As noted earlier, both $H_P$ and $H_Q$ are linear hypergraphs, since all elements in $A$ are smaller than all elements in $B$. We claim that these hypergraphs both contain a large matching with high probability.
\begin{claim}\label{claim:matchingsMPv3}
        With probability $1-o(1)$\ there is both a hypergraph matching $M_P \subseteq E(H_P)$ covering all but at most $2\mu|H_P|$ vertices of $H_P$, and a hypergraph matching $M_Q \subseteq E(H_Q)$ covering all but at most $2\mu|H_Q|$ vertices of $H_Q$.
    \end{claim}
    
  \begin{proofclaim}  
    We will only prove that with high probability the matching $M_P$ with the desired property inside $H_P$ exists, since one can then apply the same method to $H_Q$ to find the matching $M_Q$, and take a union bound to get that with high probability they both exist, as desired to prove the claim. In order to find $M_P$, we will use \cref{thm:Pippenger}. To apply this, we first need to show that the hypergraph $H_P$ satisfies various properties. We will show that the following properties hold with high probability.
    \begin{enumerate}[label = \upshape{(P\arabic*)}, leftmargin = \widthof{R100000}]
    \item $|H_P| = (1\pm \gamma)m$. \label{ob1}
            \item $H_P$ has maximum degree at most $  (1+\sqrt{\gamma})\frac{\mu p^2 }{20}|H_P|$. \label{ob2}
        \item All vertices in the set $\left(A_1 \cap \left[2s, \frac{n}{2}-s\right]\right) \cup \left(B_2 \cap \left[\frac{n}{2}+2s,n-s\right]\right) \cup \left(P_2 \cap \left[2s,n-2s\right] \right)$ have degree in the interval ${(1\pm\sqrt{\gamma})\frac{\mu p^2}{20}|H_P|}$ in $H_P$. \label{ob3}
    \end{enumerate}
    Note that $\mathbb{E}[|A_1|] =pn/2=m/3$ and hence a simple application of the Chernoff bound tells us that $|A_1|$ lies outside of the range $ (1\pm \gamma)m/3$ with probability $2e^{-\Omega(n)}$. By the same argument, the same holds for $|B_2|$ and $|P_2|$. Thus applying a union bound, \ref{ob1} holds with probability at least $ 1-6e^{-\Omega(n)}$. In order to prove \ref{ob2} and \ref{ob3} hold with high probability, we will start by considering the degrees of vertices in $H_P$. 
We know that $\Delta(G) \leq 2s$, and every colour in $G$ appears on at most $s$ edges. In particular this means that $d_R(x)\leq 2s$ for every $x \in A \cup B$, and $d_R(c) \leq s$ for every $c \in C$. By choice of $R$ we also have the additional properties that $d_R(x)= 2s$ for every $x \in\left(A \cap \left[2s, \frac{n}{2}-s\right]\right) \cup \left(B \cap \left[\frac{n}{2}+2s,n-s\right]\right) $, and $d_R(c)= s$ for every $c \in C\cap [2s,n-2s]$. 

\begin{itemize}
 \item Let $c \in C$. Consider the random variable $X_c$ that is the number of edges $abc$ in $R$ where $a\in A_1$ and $b\in B_2$. Suppose $abc \in E(R)$ for some $a\in A$ and $b\in B$. The events $a \in A_1$ and $b\in B_2$ happen independently and with probability $p$ each, so the probability that they both occur is $p^2$. Thus the expected number of edges $abc \in E(R)$ containing $c$ for which $a\in A_1$ and $b\in B_2$ is $p^2d_R(c) \leq p^2s$. In particular, we have $\mathbb{E}[X_c]\leq p^2s$. 
Furthermore, note that the pairs $a\in A_1$, $b\in B_2$ for which $c=|b-a| = b-a$ are pairwise disjoint, so that the random variables $\mathbf{1}\{abc\in E(R)\}$ over all such pairs are independent.
Applying a Chernoff bound (\cref{thm:chern}), for each $c\in C$ we have
$$ \mathbb{P}(X_c > (1+\gamma)p^2s) \leq \mathbb{P}(|X_c - \mathbb{E}[X_c]| \geq \gamma p^2s) \leq 2e^{-\frac{\gamma^2 p^2s}{3}} 
= 2e^{-\Omega(n)}.
$$
Taking a union bound over all elements of $C$, the probability that there exists $c\in C$ contained in more than $(1+\gamma)p^2s$ edges of $R$ of the form $abc$ for $a\in A_1$ and $b\in B_2$ is at most $2ne^{-\Omega(n)}$.
In particular, by definition of $H_P$, the probability that there exists a vertex $c\in C\cap V(H_P) = P_2$ with degree greater than $(1+\gamma)p^2s$ in $H_P$ is at most $2ne^{-\Omega(n)}$. 

Furthermore, for $c\in [2s,n-2s]$, we have $\mathbb{E}[X_c] = p^2d_R(c) = p^2s$ and by a Chernoff bound (\cref{thm:chern}) we have
\begin{equation*}
    \mathbb{P}(|X_c - p^2s| \geq \gamma p^2s) 
    \leq 2e^{-\frac{\gamma^2 p^2s}{3}} 
    \leq 2e^{-\Omega(n)}.
\end{equation*}
Taking a union bound over the elements of $[2s,n-2s]$, and as a consequence of the definition of $H_P$, the probability that there exists a vertex $c\in [2s,n-2s]\cap V(H_P)$ with degree outside of the range $(1 \pm \gamma)p^2s$ in $H_P$ is at most $2ne^{-\Omega(n)}$.

\item Let $b\in B$. 
	Consider the random variable $Y_b$ that is the number of edges $abc$ in $R$ where $a\in A_1$ and $c\in P_2$. Suppose $abc \in E(R)$ for some $a\in A$ and $c\in C$. The events that $a \in A_1$ and that $c=b-a\in P_2$ happen independently with probabilities $p$ and $p/2$ respectively, unless $b-a=a$, in which case there is a dependence. This latter case can only happen at most once, if $b=2a$. Thus  $\mathbb{E}[Y_b] = \frac{p^2}{2}d_R(b) \pm 1 \leq p^2s +1$. 

	For each $v \in [n]$ let $I_v$ be the random variable that is $0$ if $v \notin S_1 \cup S_2$, $1$ if $v \in S_1$, and $2$ if $v \in S_2$, and write $\Omega = \{0,1,2\}$.
	Then $Y_b$ is a $2$-Lipschitz function from $\Omega^n$ to $\mathbb{R}$. Indeed, for any $v \in [n]$, $v$ is contained in at most two edges alongside $b$ in the hypergraph $R$ (at most one for each of the possibilities that $v \in A$ and $v \in C$), so changing the outcome of whether some vertex $v$ belongs in $S_1$ or $S_2$ or neither changes the value of $Y_b$ by at most $2$.
We apply McDiarmid's inequality  (\Cref{thm:McDiarmid}) to obtain
    \begin{align}
    \begin{split}\label{eqn:mcdiarmidB_j}
    \mathbb{P}(Y_b >(1+ \gamma) p^2s) 
 & \leq \mathbb{P}(|Y_b - \mathbb{E} [Y_b] | \geq \gamma p^2s - 1) 
  \leq 2e^{-\frac{2(\gamma p^2s/2)^2}{4n}} 
 \leq 2e^{-\Omega(n)}.
    \end{split}
    \end{align}
Again by taking a union bound over all elements of $B$, we obtain that with probability at most $n e^{-\Omega(n)}$, some vertex in $B$ is contained in more than $(1+\gamma)p^2s$ edges $abc$ in $R$ with $a \in A_1$ and $c \in P_2$.
In particular, the probability that some vertex in $B_2$ has degree greater than $(1 + \gamma)p^2s$ in $H_P$ is at most $n e^{-\Omega(n)}$.

Secondly, suppose $b\in \left[\frac{n}{2}+2s,n-s\right]$. We have $\mathbb{E}[Y_b]= \frac{p^2}{2}d_R(b) \pm 1 = p^2s\pm 1$. Again since $Y_b$ is a $2$-Lipschitz function, by the same inequality from the right hand side of \cref{eqn:mcdiarmidB_j}, the probability that $Y_b \neq (1\pm \gamma) p^2s$ is at most $ 2e^{-\Omega(n)}$. 
Taking a further union bound over all elements of $[\frac{n}{2}+2s,n-s]$ and by the definition of $H_P$, this implies that with probability $2ne^{-\Omega(n)}$, there exists $b\in [\frac{n}{2}+2s,n-s]\cap V(H_P)$ with degree $ \neq (1\pm \gamma)p^2s$ in $H_P$. 

\item Let $a\in A$. Consider the random variable $Z_a$ that is the number of edges $abc$ in $R$ where $b\in B_2$ and $c\in P_2$. An identical argument to the case above gives that $\mathbb{E}[Z_a] \leq p^2s + 1$, and that $Z_a$ is a $2$-Lipschitz function from $\Omega^n$ to $\mathbb{R}$, on the same product space. Applying \cref{thm:McDiarmid}, a union bound and specialising to vertices in $A_1$ as before, with probability at most $n e^{-\Omega(n)}$, there exists some $a\in A\cap V(H_P)$ that is contained in more than $ (1+\gamma)p^2s$ in $H_P$.
Furthermore, the probability of some vertex in $[2s, \frac{n}{2}-s]\cap V(H_P)$ having degree $ \neq (1\pm \gamma)p^2s$ in $H_P$ is at most $2ne^{-\Omega(n)}$.

\end{itemize}
One further application of the union bound tells us that with probability $1-o(1)$ all vertices in $H_P$ have degree at most $(1+\gamma)p^2s$, and, all vertices in the following set have degree $(1\pm \gamma )p^2s $ in $H_P$.
\begin{equation}\label{eqn:deg_s_vxs}
    \left(A_1 \cap \left[2s, \frac{n}{2}-s\right]\right) \cup \left(B_2 \cap \left[\frac{n}{2}+2s,n-s\right]\right) \cup \left(P_2 \cap \left[2s,n-2s\right]\right).
\end{equation}
Note that, by \ref{ob1}, with high probability we have
\begin{equation*}
    (1\pm \gamma )p^2s = (1\pm \gamma )\frac{p^2\mu m }{20} = \left(\frac{1\pm \gamma}{1\pm\gamma}\right)\frac{p^2\mu}{20}|H_P|
= (1\pm \sqrt{\gamma})\frac{p^2\mu}{20}|H_P|, 
\end{equation*}
where the final inequality holds since $\frac{1+y}{1-y} \leq 1+\sqrt{y}$ and $\frac{1-y}{1+y} \geq 1-\sqrt{y}$ for all $ 0 <y \ll 1$. 

In particular, both \ref{ob2} and \ref{ob3} hold with probability $1-o(1)$, and we now seek to apply~\cref{thm:Pippenger}. 

Note that with high probability, using~\ref{ob1}, the number of vertices of $H_P$ that are outside of the set given in~\eqref{eqn:deg_s_vxs} is at most $10s = \frac{\mu}{2}m\leq \mu|H_P|$.
Taking $\sqrt{\gamma}, \mu, \ \frac{p^2\mu}{20}$ and $|H_P|$ to play the roles of $\gamma, \mu \ \alpha $ and $n$ respectively, with high probability we can apply \Cref{thm:Pippenger} to find a matching in $H_P$ covering all but $2\mu|H_P|$ vertices. Denote it by $M_P$. 
As mentioned earlier, we can apply the exact same method to $H_Q$, and apply a union bound, to get that with probability $1-o(1)$ both $M_P$ and $M_Q$ exist as in the statement of the claim. 
\end{proofclaim}

Note that $H_P$ and $H_Q$ are vertex-disjoint and $V(H_P) \cup V(H_Q) = V(H_{[n]}[S_1,S_2,S_2])$. Then with probability $1-o(1)$, there is a matching $M \coloneqq M_P \cup M_Q$ in $H_{[n]}[S_1,S_2,S_2]$ covering all but at most $2\mu(|H_P|+|H_Q|) = 2\mu|H_{[n]}[S_1,S_2,S_2]|$ vertices in $H_{[n]}[S_1,S_2,S_2]$. Since $H_{[n]}[S_1,S_2,S_2]$ is $3$-uniform and $3$-partite, then such a matching $M$ covers the same number of vertices in each of its tripartition classes, and therefore covers all but at most $2\mu\max \{|S_1|,|S_2|\}$ vertices in each class. Each vertex class has size at least $\min\{|S_1|,|S_2|\}$. Finally we take $\theta^{-1}(M)$ to be the corresponding rainbow matching in the edge-coloured graph $K_{[n]}$ with all edges between $S_1$ and $S_2$ having colour in $S_2$. Thus we know that with probability $1-o(1)$, this matching exists, and covers all but at most $2\mu\max \{|S_1|,|S_2|\}$ vertices in both $S_1$ and in $S_2$. By construction $\theta^{-1}(M) \subseteq G$ and by \Cref{lem:spanninggraph} no edge in $G$ has colour $1$. So it also follows that no edge in $\theta^{-1}(M)$ has colour $1$, as desired.
\end{proof}

\subsection{Embedding trees with a splitting vertex}

We focus on trees with a specific structure that will be particularly useful for embedding the auxiliary tree defined in the proof of \cref{lem:TminusJrooted_comps}, and as mentioned in the strategy overview in \cref{strategy}.
\begin{lemma} \label{lem:onesmallcompsv2}
    Let $n^{-1}\ll \zeta \ll \eps<1$ and suppose that $T$ is an $(n+1)$-vertex tree containing a vertex $v$ such that all components of $T- v$ have order at most $\zeta^{-1}$. Then there exists a rainbow copy of $T$ in $K_{[(1+\varepsilon)n] \cup \{0\}}$ such that $v$ receives label $0$, and no edge has colour $1$. 
\end{lemma}

\begin{proof}    Choose $\lambda \in (0,1) $ and $n \in \mathbb{N}$ to satisfy the hierarchy given by
\begin{equation}\label{5.6hierarchy}
    n^{-1} \ll \lambda \ll \zeta \ll \eps.
\end{equation}
Let $d \coloneqq \lfloor \lambda n \rfloor$. Let $k$ denote the number of all possible rooted trees on at most $\zeta^{-1}$ vertices, and denote these rooted trees as $T_1,T_2,\ldots,T_k$. Note that by Cayley's formula, there are at most $(\zeta^{-1})^{\zeta^{-1} - 1}$ trees on at most $\zeta^{-1}$ vertices. For each of these there are at most $\zeta^{-1}$ vertices to select as a root. So in total, we have $k \leq (\zeta^{-1})^{\zeta^{-1}}$. We can assume by \cref{5.6hierarchy} that $\lambda \leq \varepsilon \zeta^{\zeta^{-1}+1}/18$ and so in particular $\lambda \zeta^{-1} k \leq \varepsilon/18$. Let $T$ be an $(n+1)$-vertex tree and suppose there exists a vertex $v \in V(T)$ satisfying the assumptions of the statement. Then in particular every component of $T-v$ is isomorphic to $T_i$ for some $i \in [k]$ where the root of $T_i$ is the sole neighbour of $v$ in $T$ (there cannot be more than one neighbour of $v$ in this component as else this would create a cycle). For each $i \in [k]$, let $m_i \in \mathbb{N} \cup \{0\}$ count the number of components in $T-v$ which are isomorphic to $T_i$, with the root neighbouring $v$. We view the forest $T - v$ as a collection of rooted trees by $$T-v \cong (T_1 \times m_1) + (T_2 \times m_2) + \ldots + (T_k \times m_k).$$ 
For each $i \in [k]$, let $m_i'$ be the smallest integer at least as big as $m_i$ which is divisible by $d$. So $m_i \leq m_i' < m_i +d$.
Define $F'$ to be the forest given by $$F' = (T_1 \times m'_1) + (T_2 \times m'_2) + \ldots + (T_k \times m'_k). $$ 

It follows that $T-v$ is a subgraph of $ F'$. Let $n' \coloneqq |F'|$. Since $m_i' < m_i +d$ for every $i \in [k]$, then $n' < n+d\zeta^{-1}k \leq (1+\lambda \zeta^{-1}k)n\le (1+\varepsilon/18)n$. By our choice of each $m_i'$ being divisible by $d$, we can write $F' = \hat{F} \times d$ for the forest $\hat{F}$ given by $$\hat{F}  =  \left(T_1 \times\frac{m'_1}{d}\right)  + \left(T_2 \times\frac{m'_2}{d}\right) + \ldots +\left(T_k \times\frac{m'_k}{d}\right).$$

    So now we have that $T-v$ is a subgraph of $\hat{F} \times d $ and let us consider the tree $T'$ obtained by adding vertices and edges to $T$ so that $T' -v $ is isomorphic to $\hat{F} \times d$, where $v$ is adjacent to the root in each component of $\hat{F}$. 
    We have that $|\hat{F} \times d| = n'  \leq (1+\varepsilon/18)n$. In particular, note that
    \begin{equation}\label{eq:edges_Fhat}
        e(\hat F)< |\hat{F}| \leq (1+\varepsilon/18) \frac{n}{d}\leq (1+\varepsilon/18)\frac{2n}{\lambda n} \leq 4\lambda^{-1}.
    \end{equation}
We will require some extra labels in our complete graph in order to find a rainbow embedding of $T'$. Let $\delta \coloneqq \lambda^2$ and choose $\tilde{n}$ to be the smallest even integer that is at least $  \frac{(1+\delta)n'}{1-\delta e(\hat{F})}$. We will construct a rainbow copy of $T'$ in $K_{[\tilde{n}]\cup \{0\}}$ by mapping $v$ to $0$, and the remaining vertices in $T' \setminus \{v\}$ will be assigned a unique label in $[\tilde{n}]$ so that all edges of $T'$ have a distinct colour in $K_{[\tilde{n}] \cup \{0\}}$. We will also enforce the additional restriction that colour $1$ cannot be used. Since $T \subseteq T'$, this defines a rainbow embedding of $T$ in $K_{[\tilde{n}] \cup \{0\}}$ where $v$ receives label $0$, by considering the same embedding restricted to $V(T)$. 
Using \cref{eq:edges_Fhat} and recalling that $n'\leq (1+\varepsilon/18)n$, it follows from \cref{5.6hierarchy} that 
\begin{equation*}
    \Tilde{n} \leq \left\lceil \frac{(1+\delta)n' }{1-\delta e(\hat{F})} \right\rceil +1   
\leq \frac{(1+2\delta)(1+\varepsilon/18)n}{1-4\delta\lambda^{-1}} 
\leq (1+\varepsilon)n,
\end{equation*}

So, finding a rainbow copy of $T$ in $K_{[\tilde{n}] \cup \{0\}}$ satisfying the required properties will give us our desired embedding of $T$ in $K_{[(1+\varepsilon)n] \cup \{0\}}$.
 
Let $ p \coloneqq \frac{1}{|\hat{F}|} = \frac{d}{n'}$ and consider a $p$-random partition $[\tilde{n}] = S_1 \cup  S_2 \cup \ldots \cup S_{|\hat{F}|}$, meaning that each element of $ [\tilde{n}] $ is selected independently and uniformly at random with probability $p$ to belong to a single part. Then for each $i \in |\hat{F}|$, $|S_i| \sim \text{Bin}(\tilde{n},p)$ and $\mathbb{E}|S_i| = p\tilde{n} = \frac{(1+\delta)d}{1-\delta e(\hat{F})} \pm \frac{2d}{n'}$. So applying \cref{thm:chern} gives
\begin{align}\label{eqn:S_i_Chernoff}
    \mathbb{P}\biggl[\Bigl||S_i| - \mathbb{E}|S_i|\Bigr| > \frac{\delta d}{2(1-\delta e(\hat{F}))}\biggr]\leq \mathbb{P}\left[\Bigl||S_i| - \mathbb{E}|S_i|\Bigr| > \frac{\delta }{8}\mathbb{E}|S_i|\right]  
    & \leq 2e^{-\frac{\delta^2(1+\delta)d}{192(1-\delta e(\hat{F}))}} = 2e^{-\Omega(n)}.
\end{align}
Hence, with probability $1 - 2ne^{-\Omega(n)}$, we have $|S_i| = \frac{(1+\delta)d}{1-\delta e(\hat{F})} \pm  \left(\frac{\delta d}{2(1-\delta e(\hat{F}))} + \frac{2d}{n'}\right) = \frac{(1+\delta)d}{1-\delta e(\hat{F})} \pm  \frac{\delta d}{1-\delta e(\hat{F})}$ for all $i \in [|\hat{F}|]$.

Let $\mu \coloneqq \frac{\delta\zeta(1 - 2\delta)}{8(1+2\delta)}$ and note that $n^{-1}\ll \mu$. We have $p = d/n'\geq \lambda/4$, and so $n^{-1}\ll p$.
Consider any pair $i,j \in [|\hat{F}|]$ such that $i<j$. 
Apply \cref{lem:mathcingSiSjv3} with $S_i,S_j,\Tilde{n}, p$ and $\mu$ playing the roles of $S_1,S_2,n, p$ and $\mu$ respectively, to deduce that with probability $1-o(1)$, there is a rainbow matching with edges between $S_i$ and $S_j$ having colours in $S_j \setminus \{1\}$, covering all but at most $2\mu\max\{|S_i|,|S_j|\}$ vertices in both $S_i$ and $S_j$. Taking $n$ sufficiently large and applying a union bound, 
we deduce that there is a partition $[\tilde{n}] = S_1\cup \dots \cup S_{|\hat{F}|}$ satisfying the properties that
\begin{equation}\label{eqn:sizeS_i}
    \frac{d}{1-\delta e(\hat{F})}\leq |S_i| \leq \frac{(1+2\delta)d}{1-\delta e(\hat{F})}
\end{equation} 
for all $i \in [|\hat{F}|]$, and that there is a collection of matchings $\mathcal{M}_0 \coloneqq \{M_{ij} : i,j \in [|\hat{F}|], i <j\}$ such that each $M_{ij}$ is $S_j$-rainbow, covers all but at most $2\mu\max\{|S_i|,|S_j|\}$ vertices in both $S_i$ and $S_j$, with all edges between $S_i$ and $S_j$, and that no edge in any matching in this collection has colour $1$. Fix such a partition for the remainder of the proof. 

We will combine these matchings in order to embed the graph $\hat{F}\times d$. Enumerate the vertices of $\hat{F}$ as $u_1,\ldots,u_{|\hat{F}|}$ using the rooted ordering of each component, that is, all roots are given the lowest indices, and vertices are arbitrarily labelled in ascending order of distance from a root.
Note that any vertex in this ordering has at most one neighbour amongst the lower indexed vertices, using the property that rooted trees are $1$-degenerate.  
 Let $\mathcal{M} \coloneqq \{M_{ij} : u_iu_j \in e(\hat F)\}$. 

Order the copies of $\hat{F}$ in $\hat{F}\times d$ arbitrarily and denote the vertices of $\hat{F}\times d$ by the set $\{u_j^s: j\in [|\hat{F}|], s\in [d]\}$, where $u_j^s$ corresponds to the vertex $u_j$ belonging to the $s$th copy of $\hat{F}$.
We use the following claim. 

\begin{claim}\label{claim:rainbowforest}
 $K_{[\tilde{n}]}$ contains a rainbow copy of $\hat{F} \times d$ such that if $\psi : V(\hat{F} \times d) \rightarrow [\Tilde{n}]$ corresponds to this embedding, then the following properties are satisfied:
\begin{enumerate}[label = \upshape{(\roman*)}, leftmargin = \widthof{ABC00}]
	\item \label{itm:claim-1}
for every $j \in [|\hat{F}|]$ and $s\in [d]$, we have $\psi(u_j^s) \in S_j$,
\item \label{itm:claim-2}
	for every $\ell \in [|\hat{F}|]$ such that $u_{\ell}$ is a root in a component of $\hat{F}$,  
no edge in the embedding has a colour in $S_{\ell}$,
\item \label{itm:claim-3}
	no edge in the copy has colour $1$ and no vertex gets embedded to $1$, i.e.\ $1\notin \textnormal{im}(\psi)$.
\end{enumerate}
\end{claim}

\begin{proofclaim}
Recall that a rainbow copy of $\hat{F} \times d$ is exactly the union of $e(\hat{F})$ colour-disjoint rainbow matchings, where each matching has size $d$ and corresponds to a unique edge in $\hat{F}$. So, we use the partition $S_1 \cup \ldots \cup S_{\hat{F}}$ to find a collection of rainbow matchings, each of which contains only edges of colour in $S_j$ for some $j \in [|\hat{F}|]$, and such that no two matchings use the same colour set.

Note that by choice of our ordering, for a fixed index $j$, there is at most one $i<j$ such that  $u_iu_j \in E(\hat{F})$. In particular at most one matching in $\mathcal{M}$ is $S_j$-rainbow. Since this holds for all $j$, then the matchings in $\mathcal{M}$ are pairwise colour-disjoint in $K_{[\tilde{n}]}$, and thus clearly also pairwise edge-disjoint. 

For each $t\in [|\hat{F}|]$, let us define the subcollection of matchings $\mathcal{M}_t \coloneqq \{M_{ij} \in \mathcal{M}: i = t \text{ or } j=t \}.$
Consider the set of `bad' vertices in $S_t$ to be those in the set $$B_t \coloneqq \{1\} \cup \left( S_t \setminus \bigcap_{M\in \mathcal{M}_t}V(M) \right).$$ 
By \cref{eqn:sizeS_i}, it follows that $|S_i| \geq |S_j| -\frac{2\delta d}{1-\delta e(\hat{F})}$ and vice versa for any pair $i,j$. Without loss of generality assume that $|S_j| \geq |S_i|$ (else we can easily apply the same argument the other way around), and we consider the lower bound in \cref{eqn:sizeS_i} to see that
\begin{equation*}
2 \mu|S_j| = \frac{\delta\zeta(1 - 2\delta)}{4(1+2\delta)}|S_j| = \frac{\delta\zeta}{4(1+2\delta)}\left(|S_j| - 2\delta|S_j| \right) \leq \frac{\delta\zeta}{4(1+2\delta)} \left(|S_j| - \frac{2\delta d}{1-\delta e(\hat{F})} \right) 
\leq \frac{\delta\zeta}{4(1+2\delta)}|S_i|.
\end{equation*}
So we can conclude that $ 2\mu\max\{|S_i|,|S_j|\} \leq \frac{\delta\zeta}{4(1+2\delta)}\min \{|S_i|,|S_j|\}$, and therefore every matching $M_{ij} \in \mathcal{M}$ covers all but at most $\frac{\delta\zeta}{4(1+2\delta)} \min\{|S_i|,|S_j|\}$ vertices in $S_i$ and all but at most $\frac{\delta\zeta}{4(1+2\delta)} \min\{|S_i|,|S_j|\}$ vertices in $S_j$.

In particular, 
for every $i\in [|\hat{F}|]$ we have $|B_i \setminus \{1\}|\leq \sum_{M\in \mathcal{M}_i}|S_i\setminus V(M)|\leq |\{e\in E(\hat{F}): u_i\in e\}|\cdot \frac{\delta\zeta}{4(1+2\delta)} |S_i|$.
Again using \cref{eqn:sizeS_i}, and since $1 \in B_t$ for exactly one $t\in [|\hat{F}|]$,  we deduce that \begin{equation}\label{eq:sum_bad_vertices}
    \sum_{i\in [|\hat{F}|]}|B_i|\leq 1 +\sum_{i\in [|\hat{F}|]}\sum_{\substack{e\in E(\hat{F}):\\ u_i \in e}} \frac{\delta\zeta}{2(1+2\delta)}|S_i| 
    \leq \frac{\delta\zeta}{1+2\delta} e(\hat{F})\max_{j}|S_j|\leq \delta \zeta e(\hat{F})\min_{j}|S_j|.
\end{equation}
Let $H \subseteq K_{[\Tilde{n}]}$ be the graph obtained by taking the union of the matchings $\bigcup_{M \in \mathcal{M}} M$ and deleting all connected components which contain a bad vertex. The components of $H$ combine to form copies of $\hat{F}$, of which each copy will contain exactly one vertex in each part.
Since every component of $\hat{F}$ has at most $\zeta^{-1}$ vertices and using \cref{eq:sum_bad_vertices}, then in total we delete at most $\zeta^{-1}\sum_{i}|B_i| \leq \delta e(\hat{F})|S_t|$ vertices from $\bigcup_{M \in \mathcal{M}} M$ within a given part $S_t$, in order to obtain $H$. So, using the lower bound from \cref{eqn:sizeS_i}, we can find at least $(1-\delta e(\hat{F}))|S_t| \geq  d$ copies of $\hat{F}$ in total. Choosing exactly $d$ of them we find an embedding of $\hat{F} \times d$ in $K_{[\Tilde{n}]}$ which we have shown is rainbow.

No edges in the matchings in $\mathcal{M}$ have colour $1$, and by choice of $B_t$ we have $1\notin V(H)$, and so property \ref{itm:claim-3} holds. The way in which we have constructed this forest means that each matching $M_{ij}$ corresponds to an edge $u_iu_j \in E(\hat{F})$, and is used for the image of edges of the form $u_i^su_j^s$. So, all of the endpoints belong in the corresponding vertex classes of the partition, i.e.\ for every $j \in [|\hat{F}|]$, each copy of $u_j$ is embedded into the part $S_j$, as desired for property \ref{itm:claim-1}. For any root $u_{\ell}$ in some component of $\hat{F}$, $u_{\ell}$ has no neighbour in $\hat{F}$ with a lower index in the ordering $u_1,\ldots,u_{|\hat{F}|}$, and so no matching of the form $M_{k\ell}$ with $k<\ell$ has been used to construct the copy of $\hat{F} \times d$. This means there is no $S_{\ell}$-coloured matching, and so property \ref{itm:claim-2} holds. Thus all conditions of \cref{claim:rainbowforest} are satisfied.
\end{proofclaim}

So, now let us use the construction from the claim to find our desired copy of $T$ in $K_{[\Tilde{n}] \cup \{0\}}$. First we find a copy of $T'$. Let us extend the mapping $\psi$ of $V(\hat{F} \times d)$ given by \cref{claim:rainbowforest}, by adding the assignment $\psi(v) \coloneqq 0$. Then $\psi:V(T') \rightarrow [\Tilde{n}] \cup \{0\}$ is an injection providing us with an embedding of $T'$ into $K_{[\Tilde{n}] \cup \{0\}}$. We aim to prove this is rainbow. Since we already know $\psi$ restricted to $V(\hat{F} \times d)$ is rainbow, it remains to check that by mapping $v$ to the label $0$, we create no colour repetitions.

By choice of $\hat{F}$, we know that only the roots within each component of $\hat{F}$ are neighbours of $v$, and thus by embedding $v$ to $0$ in $K_{[\tilde{n}]\cup \{0\}}$, we only add edges between $v$ and copies of roots in $\hat{F}$. These edges will have colour $|\psi(u_{\ell}^s) - \psi(v)| = |\psi(u_{\ell}^s) - 0| =\psi(u_{\ell}^s)$ for some copy $u_\ell^s$ of a root $u_{\ell} \in V(\hat{F})$, and since we know $\psi(u_\ell^s) \in S_{\ell}$ from property \ref{itm:claim-1} and $\psi(u_\ell^s) \neq 1$ from property \ref{itm:claim-3}, then all colours added in this embedding belong in $S_{\ell}$ for some $\ell$ such that $u_{\ell}$ is a root, and none of the edges have colour $1$. Also by property \ref{itm:claim-2} of the embedding, these edges will be colour-disjoint from those already embedded within $\psi(\hat{F}\times d)$. Our resulting copy of $T'$ is therefore rainbow.
Finally, we restrict this embedding by only considering the vertices of $T$, which is possible since $T$ is a subtree of $T'$. This gives us our desired rainbow copy of $T$ in $K_{[\tilde{n}] \cup \{0\}} \subseteq K_{[(1+\varepsilon/2)n ]\cup \{0\}}$, such that $v$ is assigned to label $0$, and we have avoided using colour $1$ on any edge.
\end{proof}

\subsection{Extending the rainbow embedding}

We hope to use the statement of \cref{lem:onesmallcompsv2} in order to consider more families of trees. We do this by proving that for certain trees, it is possible to contract vertex sets so that the resulting tree contains a vertex of high degree which separates it into components of small order. Applying \cref{lem:onesmallcompsv2}, the labelling provided will then be extended by converting each edge of the contracted tree into a matching in the original. First, we need the following proposition which will be useful in finding rainbow perfect matchings between intervals.

  \begin{proposition}\label{prop:matchingintervals}
Let $i, j, \ell, n \in \mathbb{N}$ be such that $\ell$ is odd, $i<j$ and $(j+1)\ell \leq n$. Consider two disjoint intervals of integers $I_i = \{i\ell +1,\ldots,(i+1)\ell\}$ and $ I_j = \{j\ell +1,\ldots,(j+1)\ell\}$, and an interval of colours,  
$C_{ij}=\left\{(j-i)\ell-\lfloor \ell/2 \rfloor, \ldots, (j-i)\ell+\lfloor \ell/2 \rfloor\right\}$. Then there exists a rainbow perfect matching in $K_{[n]}$ from $I_i$ to $I_j$ using colours in $C_{ij}$.
 \end{proposition}

\begin{proof}
Note that $\ell = \lceil \ell/2 \rceil + \lfloor \ell/2 \rfloor$. We match up the first $\lceil \ell/2 \rceil$ elements of $I_i$ with the first $\lceil \ell/2 \rceil$ elements of $I_j$, by smallest to largest, second smallest to second largest, and so on, to get a matching $M_1$ explicitly defined as
    $$M_1 = \left\{ \left (i\ell + y,j\ell + \left\lceil\ell/2\right\rceil + 1 -y\right) : y \in \left\{1,2,\ldots, \left\lceil\ell/2\right\rceil \right\}  \right\}. $$
    Similarly we match up the last $\lfloor \ell/2 \rfloor$ elements of $I_i$ with the last $\lfloor \ell/2 \rfloor$ elements of $I_j$ in the same way to get a matching $M_2$ given by
$$M_2 = \left\{ \left (i\ell+\left\lceil\ell/2\right\rceil + z,(j+1)\ell +1 -z\right) : z \in \left\{1,2,\ldots, \left\lfloor\ell/2\right\rfloor \right\}  \right\}.$$

\begin{figure}[b]
    \centering
    \begin{tikzpicture}[scale=1]
    \draw[line width =  0.5mm,blue] (1,1) -- (4,-0.7);
    \draw[line width =  0.5mm,teal] (2,1) -- (3,-0.7);
    \draw[line width =  0.5mm,green] (3,1) -- (2,-0.7);
    \draw[line width =  0.5mm,cyan] (4,1) -- (1,-0.7);
     \draw[line width =  0.5mm,red] (5,1) -- (7,-0.7);
    \draw[line width =  0.5mm,orange] (6,1) -- (6,-0.7);
    \draw[line width =  0.5mm,pink] (7,1) -- (5,-0.7);
    \foreach \i in {1,2,3,4,5,6,7}{
    \filldraw (\i,1) circle (2.5pt) node[anchor=south, yshift=0.2em]{\i};}
    \foreach \j in {15,16,17,18,19,20,21}{
    \filldraw (\j -14,-0.7) circle (2.5pt) node[anchor=north, yshift=-0.2em]{\j};}
    \draw [decorate,decoration={brace,amplitude=10pt,mirror,raise=1ex}]
  (0.85,-1) -- (4.15,-1) node[midway,xshift= -0.1em,yshift=-2.2em]{\footnotesize $M_1$ contains odd colours};
  \draw [decorate,decoration={brace,amplitude=10pt,mirror,raise=1ex}]
   (4.85,-1.05) -- (7.15,-1) node[midway,xshift= 0.1em, yshift=-2.2em]{\footnotesize $M_2$ contains even colours};
         \end{tikzpicture}
\caption{Rainbow matching between intervals $I_0$ and $I_2$ for $\ell = 7$, with colour set $\{11,12,\ldots,17\}$.}
    \label{fig:interval_matching}
\end{figure}
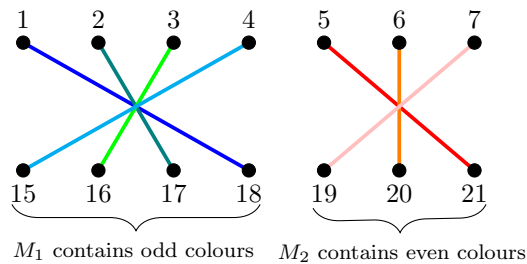

See \cref{fig:interval_matching} for an example. Note that $V(M_1) \cup V(M_2) = I_i \cup I_j$ and $V(M_1) \cap V(M_2) =\emptyset $, thus $ M \coloneqq M_1 \cup M_2$ is a perfect matching. It is easy to see that both $M_1$ and $M_2$ are individually rainbow, since the colour of the edges strictly decrease as we match up the pairs from smallest to largest and so on. It remains to verify that their union maintains the rainbow property, and all edges have 
a colour in $C_{ij}$.

Consider any edge in $M_1$. It will be in the form $ \left (i\ell + y,j\ell+ \lceil\ell/2\rceil +1 -y\right)$ for some $y \in \left\{1,2,\ldots,  \lceil\ell/2\rceil \right\}$. Calculating the absolute difference, the colour of this edge is
$$j\ell + \left\lceil\ell/2\right\rceil +1 -y - i\ell-y = (j-i)\ell + \left\lceil\ell/2\right\rceil -2y +1.$$
Therefore $C(M_1) = \left\{(j-i)\ell + \left\lceil \ell/2\right\rceil -2y +1 : y\in \left\{1,2,\ldots,  \left\lceil\ell/2\right\rceil \right\}\right\}$. By the same reasoning, we similarly deduce that $C(M_2)=\left\{(j-i)\ell +\left\lceil\ell/2\right\rceil - 2z : z\in \left\{1,2,\ldots,  \left\lfloor \ell/2 \right\rfloor \right\}\right\}$, using the identity $\ceil{\ell/2} = \ell - \ceil{\ell/2} +1$ since $\ell$ is odd. Colours in each of these sets have different parities, and thus they are disjoint. So, $M$ is a rainbow perfect matching such that each of its edges has a unique colour in the set
$$ C(M_1)\cup C(M_2) =
\left\{ (j-i)\ell + \left\lceil\ell/2\right\rceil - w : w \in \{1,2,\ldots, \ell\}  \right\},
$$
and this is equal to $C_{ij}$ as desired.
 \end{proof}

\begin{lemma}\label{lem:TminusJrooted_comps}
    Let $m^{-1} \ll \lambda\ll \zeta \ll \eps \ll 1$ be such that $\zeta m \in \N$, and let $T$ be an $m$-vertex forest. Suppose there exists a set of vertices $S \subseteq V(T)$ satisfying $|S|\leq \lambda m$ and $T \setminus S \cong F \times \zeta m$ for some forest $F$ with rooted components, such that within each component the root has at most one neighbour in $S$, and all other vertices in the component have no neighbour in $S$. Then $K_{[(1+\varepsilon)m]}$ contains a rainbow copy of $T$.
\end{lemma}

\begin{proof}
    Choose $\gamma \in (0,1)$ such that $ 1/\gamma \in \mathbb{N}$ and $\zeta/\gamma \in \mathbb{N}$, and so that the following hierarchy is satisfied
\begin{equation}\label{hierarchyTnotJ}
m^{-1} \ll \lambda\ll \gamma \ll \zeta \ll \eps \ll 1.
\end{equation}
Let $T$ be as in the statement of the lemma. Without loss of generality we may assume $T$ is a tree and that for every component of $T\setminus S$, the root has exactly one neighbour in $S$, by adding edges between roots of components and vertices in $S$, and edges within $S$ if necessary. Let $d \in \{\lceil \gamma m\rceil, \lceil \gamma m\rceil +1\}$ be such that $d+3|S|$ is odd. We will first construct an auxiliary tree $T_{\text{aux}}$ obtained from $T$, such that $T_{\text{aux}}$ contains a vertex $v$ satisfying the assumptions of \Cref{lem:onesmallcompsv2}, allowing us to find a rainbow embedding of $T_{\text{aux}}$ in the difference coloured complete graph on a little more than $|T_{\text{aux}}|$ vertices.

Let us denote $V(F) \coloneqq \{u_1,\ldots, u_{|F|}\}$ using the rooted ordering of $F$, that is, all roots in the components of $F$ have lowest index, and the remainder are enumerated in ascending order in terms of increasing distance from a root. Since in $T \setminus S$ we have exactly $\zeta m$ copies of each vertex in $F$, then we can denote  the set of vertices in $T \setminus S$ by 
$$V(F \times \zeta m) = \{u_{\ell}^t : \ell \in [|F|], t\in [\zeta m]\}.$$

Let $\var : = \zeta \gamma^{-1}$. We partition $V(F \times \zeta m)$ into $\var|F|$ disjoint vertex classes, each of which will correspond to a vertex in $T_{\text{aux}}$. For this purpose, for each $\ell \in [|F|] $ and $r\in [\var]$, define
\begin{equation*}
    U_\ell^r \coloneqq \left\{u_{\ell}^t : t \in ((r-1)d,rd ]\right\}.
\end{equation*}
In particular,  for all $\ell \in [|F|]$ we have $|U^r_\ell|=d$ when $r\in [\var-1]$ and $|U^{\var}_\ell|\leq d$.
Let us construct an auxiliary graph $F_{\text{aux}}$ on these $U_\ell^r$ sets that is isomorphic to $F \times \var$. We take $F_{\text{aux}}$ to be the union of $\var$ vertex-disjoint copies of $F$ where for each $r \in [\var]$, the vertices in the $r$th copy of $F$ are given by $\{U_\ell^r : \ell \in [|F|]\}$. More explicitly, we have
\begin{equation*}
V\left(F_{\text{aux}}\right) = \{U_\ell^r : \ell \in [|F|], r \in [\var]\} \quad \text{and} \quad E\left(F_{\text{aux}}\right) = \{U_i^rU_j^r : r\in  [\var],  u_iu_j \in E(F), i,j \in [|F|]\}.
\end{equation*}   
Let $T_{\text{aux}}$ be the tree obtained by adding a new vertex $v$ to $F_{\text{aux}}$, and adding an edge between $v$ and the root vertex in every component of $F_{\text{aux}}$. We can write this explicitly as
\begin{equation*}
    V(T_{\text{aux}}) = V(F_{\text{aux}}) \cup \{v\} \quad \text{and} \quad E(T_{\text{aux}}) = E(F_{\text{aux}}) \cup \{vU_i^r :  r\in [\var], \ u_i \text{ is a root in a component of $F$}\}.
\end{equation*}
    Note that since in $T$ each component of $F \times \zeta m$ had at most one edge going into $S$ by assumption, then we do not form any cycles nor multi-edges in $T_{\text{aux}}$, so $T_{\text{aux}}$ is indeed a tree.  Also, $|F| \leq \frac{m}{\zeta m} = \zeta^{-1}$ so in particular every component of $F$ has size at most $\zeta^{-1}$. We know that $|F|\leq \zeta^{-1}$ and so $|T_{\text{aux}}| \leq \gamma^{-1} +1$. On the other hand, $|\taux| > h|F| \geq \zeta \gamma^{-1}\cdot \frac{1-\lambda}{\zeta} \geq \gamma^{-1}/2 \gg \zeta^{-1}$.    
Since all components of $T_{\text{aux}} \setminus\{v\}$ have size at most $\zeta^{-1}$, then we can apply \Cref{lem:onesmallcompsv2} to $T_{\text{aux}}$ with $\eps/2$ playing the role of $\eps$ to obtain a rainbow embedding $\phi :V(T_{\text{aux}}) \rightarrow \{0, 1, 2, \ldots, (1
+\varepsilon/2)\gamma^{-1}\}$ in $K_{[(1
+\varepsilon/2)\gamma ^{-1}]\cup\{0\}}$ where $\phi(v) = 0$ and we avoid colour $1$. 
We will use this rainbow embedding $\phi$ of $T_{\text{aux}}$ to construct a rainbow embedding $\psi$ of $T$ in $K_{[(1+\varepsilon)m]}$. For convenience, let $\eta \coloneqq (1+\varepsilon/2)\gamma^{-1}$.

We now look only at the labels of $V(F\times \var)$ given by $\phi$, and note that these all belong to $[\eta]$.
Let $\Tilde{m} \coloneqq (\eta +1)(d+3|S|)$ and consider a partition of $[\Tilde{m}]$ into $\eta+ 1$ disjoint intervals, each of length $d+3|S|$.  We denote this by $[\Tilde{m}] = I_0 \cup I_1 \cup \ldots \cup I_{\eta}$ such that for each $i \in \{0,1,\ldots,\eta\}$, 
    $$I_i \coloneqq \left\{i(d+3|S|) +1, \ldots, (i+1)(d+3|S|) \right\}.$$ 

Consider the function $g: V(T) \rightarrow \{I_0, I_1,\ldots,I_{\eta}\}$ defined for each $w \in V(T)$ as follows:
If $w \in S$ then let $g(w) = I_0$, otherwise $w \in V(F\times \zeta n)$ and so there exists a unique pair $(\ell,t) \in [|F|] \times [\zeta m]$ such that $w = u^t_{\ell}$. In this case let $g(w) = I_{\phi(U^r_{\ell})}$ for the unique $r$ with $t \in ((r-1)\gamma m, r\gamma m]$ (or equivalently, for the unique $r$ with $w\in U_{\ell}^r$). We will use $g$ to show that there exists a rainbow embedding $\psi$ of $T$ where every vertex $w \in V(T)$ is assigned a label from the interval $g(w)$. We embed $V(T)$ into $K_{[\Tilde{m}]}$ by embedding all vertices in these vertex sets in the following order:
\begin{equation}\label{eqn:ordering}
    S, U_1^1,U_1^2,\ldots, U_1^{\var},U_2^1,\ldots, U_2^{\var},\ldots,U_{|F|}^1, \ldots,U_{|F|}^{\var}.
\end{equation}
We will ensure the colours used on edges contained in $T[S]$ belong in the interval $C_S \coloneqq [3|S|]$, and that all remaining edges outside of $T[S]$ receive colours chosen from some other disjoint intervals. For this purpose, for every pair $i,j~\in~[
    \eta]~\cup~\{0\}$ with $j>i$, define the colour set $$C_{ij} \coloneqq \left\{(j-i)(d+3|S|)-\left\lfloor\frac{d+3|S|}{2}\right\rfloor, \ldots, (j-i)(d+3|S|)+\left\lfloor\frac{d+3|S|}{2}\right\rfloor\right\}.
    $$

\textbf{Embedding vertices in $\boldsymbol{S}$.} We enumerate the vertices of $S$ as $w_1,\ldots,w_{|S|}$ such that each of these has at most one neighbour amongst the lower indexed vertices, possible since $T[S] \subseteq T$ is $1$-degenerate. Each vertex $w_i \in S$ satisfies $g(w_i) = I_0$, and we want to embed 
$S$ into a subinterval $I_0' \subset I_0$, defined by 
$$I_0' = \left\{\left\lceil\frac{d-3|S|}{2}\right\rceil, \ldots, \left\lceil\frac{d+3|S|}{2}\right\rceil \right\},$$
and such that the resulting embedding $\psi$ is rainbow.
We do this greedily. Suppose we have labelled the first $j$ vertices in this ordering for some $ 0 \leq j < |S|$ and we want to assign a label to $w_{j+1}$. There are $j <|S|$ labels from $I_0'$ unavailable due to previously embedded vertices in $S$. Since at each step, there is at most one new edge having both of its endvertices already labelled, then we have also used at most $j$ colours on these edges. Each colour used causes at most two labels from $I_0'$ to become unavailable by considering the absolute difference, in total restricting $2j < 2|S|$ colours. Since $|I_0'| \geq 3|S|$ then there are suitable labels remaining, and we select $\psi(w_j)$ to be the minimal label from this set of choices. Note that $I_0'$ is an interval of length $3|S|$, and so all colours used in this process for edges in $T[S]$ belong in $C_S = [3|S|]$. Furthermore under $\psi$, the forest $T[S]$ is rainbow.
    
\textbf{Embedding root vertices.} We proceed by embedding the vertices which are roots in some component of $F \times \zeta m$, noting that these sets came first (after $S$) in the ordering \cref{eqn:ordering}. Consider a vertex set $U_\ell^r$ where all vertices in this set are copies of $u_\ell$ for some $\ell \in [|F|]$ such that $u_{\ell}$ is a root. 
Recall that we are assuming every copy of a root vertex has exactly one neighbour in $S$ in $T$, and that $u_\ell$ is a neighbour of $v$ in $T_{\text{aux}}$. Choose $j\coloneqq \phi(U_\ell^r)$. We will embed $U_\ell^r$ into a subinterval $I_j'\subset I_j$, adding only edges with colour in $C_{0j}$. We define
    $$
    I_j' \coloneqq  \left\{j(d+3|S|) +1, \ldots, (j+1)(d+3|S|) - 3|S|\right\}.
    $$
Let $d'=|U^r_\ell|$ and recall that $d'\leq d$. Just for this step, consider a new enumeration on the vertices in $U_\ell^r$ by $v_1,v_2,\ldots,v_{d'}$, chosen such that, for the unique neighbour $s_i \in N_T(v_i)\cap S$, we have
\begin{equation*}
    \psi(s_1)\geq \psi(s_2)\geq \ldots \geq \psi(s_{d'}).
\end{equation*}
Extend the embedding $\psi$ so that the vertices of $U_\ell^r$ are embedded into $I_j'$ using the ordering $v_1,\ldots,v_{d'}$, so that $v_1$ receives the smallest label in $I_j'$, $v_2$ receives the second smallest label in $I_j'$, and so on. This is possible because $d' \leq d = |I'_j|$. Let $i,k\in [\gamma m]$ be such that $k>i$. Then $\psi(v_i) - \psi(s_i) \leq \psi(v_i) - \psi(s_k) < \psi(v_k) - \psi(s_k).$ So, the colour of the edge $v_is_i$ is strictly smaller than the colour of the edge $v_ks_k$, under the embedding $\psi$. Since this holds for all such $i,k$, then any pair of edges added in the process of embedding $U_\ell^r$ have different colours. See \cref{fig:intervals_semi} for an example.

We also wish to verify that these new edges have colours in $C_{0j}$.
It suffices to check the colour of an edge $uv$ where $u\in I_0'$ and $v\in I_j'$. We have $v-u \leq (j+1)(d+3|S|) - 3|S| - \left\lceil\frac{d-3|S|}{2}\right\rceil = j(d+3|S|)+\left\lfloor\frac{d+3|S|}{2}\right\rfloor$, and $v-u\geq j(d+3|S|)+1 -  \left\lceil\frac{d+3|S|}{2}\right\rceil = j(d+3|S|) -\left\lfloor\frac{d+3|S|}{2}\right\rfloor.$ So $v-u \in C_{0j}$, as desired. This concludes our embedding process for vertices which have neighbours in $S$.

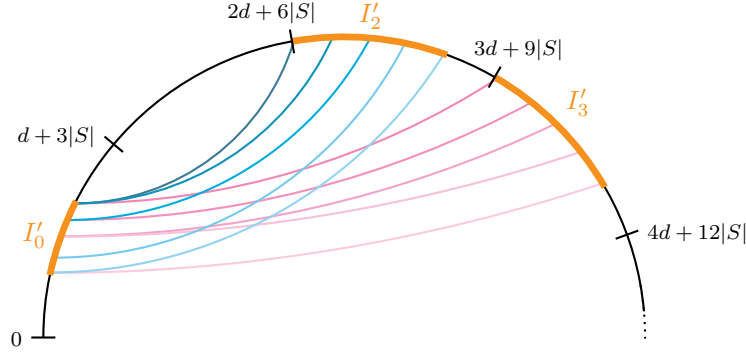
\begin{figure}
    \centering
    \begin{tikzpicture}[scale= 0.8]
    \draw[black, thick] (-5,-2) arc (180:5:5);
    \draw[black, thick,dotted] ({-5*cos(170)}, {5*sin(170) -2}) arc (10:-0.2:5);
    
                \draw[magenta!60!white, thick] ({-5*cos(26.5)}, {5*sin(26.5) -2}) arc (90:123:-12.9);
                \draw[magenta!55!white, thick] ({-5*cos(23)}, {5*sin(23) -2}) arc (90:118.5:-16.2);
                \draw[magenta!45!white, thick] ({-5*cos(19.7)}, {5*sin(19.7) -2}) arc (90:115.6:-19);
                \draw[magenta!30!white, thick] ({-5*cos(19.6)}, {5*sin(19.6) -2}) arc (90:108.6:-27.2);
                \draw[magenta!25!white, thick] ({-5*cos(12.4)}, {5*sin(12.4) -2}) arc (90:108.5:-29);
                
                \draw[cyan!50!black, thick] ({-5*cos(26.5)}, {5*sin(26.5) -2}) arc (90:164:-3.8);
                \draw[cyan!70!black, thick] ({-5*cos(26.50)}, {5*sin(26.5) -2}) arc (90:155.5:-4.7);
                \draw[cyan!90!black!90!white, thick] ({-5*cos(23)}, {5*sin(23) -2}) arc (90:152:-5.7);
                \draw[cyan!90!black!50!white, thick] ({-5*cos(15.4)}, {5*sin(15.4) -2}) arc (90:153:-6.55);
                \draw[cyan!90!black!40!white, thick] ({-5*cos(12.5)}, {5*sin(12.5) -2}) arc (90:148.5:-7.7);
            \draw[BurntOrange, line width=0.9mm] ({-5*cos(12)}, {5*sin(12) -2}) arc (-12:-27:-5); 
                 \node[BurntOrange, anchor= east] at ({-5*cos(20) -.1}, {5*sin(20) -2}) {$I_0'$};
            \draw[BurntOrange, line width=0.9mm] ({-5*cos(80)}, {5*sin(80) -2}) arc (100:70:5);
                \node[BurntOrange,anchor=south] at ({-5*cos(95)}, {5*sin(95) -2}) {$I_2'$};
            \draw[BurntOrange,line width=0.9mm] ({-5*cos(120)}, {5*sin(120) -2}) arc (60:30:5);
                \node[BurntOrange,anchor=south west] at ({-5*cos(135)}, {5*sin(135) -2}) {$I_3'$};
    \draw[black,thick] (-5.2,-2) -- (-4.8,-2);
    \draw[black,thick] ({-5.2*cos(40)}, {5.2*sin(40) -2}) -- ({-4.8*cos(40)}, {4.8*sin(40) -2});
    \draw[black,thick] ({-5.2*cos(80)}, {5.2*sin(80) -2}) -- ({-4.8*cos(80)}, {4.8*sin(80) -2});
    \draw[black,thick] ({-5.2*cos(120)}, {5.2*sin(120) -2}) -- ({-4.8*cos(120)}, {4.8*sin(120) -2});
    \draw[black,thick] ({-5.2*cos(160)}, {5.2*sin(160) -2}) -- ({-4.8*cos(160)}, {4.8*sin(160) -2});
            \node[anchor=east] at (-5.2,-2) {\textcolor{white}{00}\footnotesize{$0$}};
        \node[anchor=east] at ({-5.2*cos(40)}, {5.2*sin(40) -2}) {\footnotesize{$d+3|S|$}};
        \node[anchor=south] at ({-5.2*cos(80)-0.3}, {5.2*sin(80) -2.05}) {\footnotesize{$2d+6|S|$}};
        \node[anchor=south] at ({-5.2*cos(120) +0.3}, {5.2*sin(120) -2.05}) {\footnotesize{$3d+9|S|$}};
         \node[anchor=west] at ({-5.2*cos(160)}, {5.2*sin(160) -2.05}) {\footnotesize{$4d+12|S|$}};
    \end{tikzpicture}
        \caption{Intervals $I_0,\ldots,I_3$ depicted around the semi-circle, with subintervals $I_0'$, $I_2'$ and $I_3'$ highlighted in orange. Vertices in $S$ are embedded into $I_0'$. In this example, $v$ is adjacent to the vertices labelled by $2$ and by $3$ in $T_{\text{aux}}$, and we embed the root vertices corresponding to these vertex sets into $I_2'$ and $I_3'$ respectively. Blue edges have a colour in $C_{02}$, and lighter blue edges have larger colours than darker blue edges. The same holds for pink edges having colours in $C_{03}$.}
    \label{fig:intervals_semi}
\end{figure}

\textbf{Embedding non-root vertices.} In order to embed all remaining vertices, we will require some rainbow matchings. For all $i,j\in [\eta]$ with $j>i$, note that $\tilde{m} \geq (j+1)(d+3|S|)$ and we chose $d$ such that $d+3|S|$ is odd, so applying \Cref{prop:matchingintervals} with
    $ \tilde{m}$ and $d+3|S|$ playing the roles of 
    $n$ and $\ell$ respectively, there exists a rainbow perfect matching in $K_{[\tilde{m}]}[I_i,I_j,C_{ij}]$. Denote this matching by $M_{ij}$. Since these exist for each pair $i = \phi(U_{k}^r),j = \phi(U_{\ell}^r)$ with $U_k^rU_\ell^r \in E(T_{\text{aux}})$, we can consider a collection of such matchings to uniquely define the labels in our embedding of $T$. 
    
Now consider a vertex set $U_\ell^r$ which does not contain vertices which have neighbours in $S$, and assume we have defined the embedding $\psi$ for all previous vertex sets from \cref{eqn:ordering}. By our chosen ordering, $U_\ell^r$ has exactly one neighbour in $T_{\text{aux}}$ amongst earlier vertex sets from the ordering \cref{eqn:ordering}, say $U_{k}^{r}$. By construction we know that $\psi(U_{k}^{r}) \subseteq I_{\phi(U_{k}^{r})}$, and we will embed $U_\ell^r$ into $I_{\phi(U_{\ell}^{r})}$ using the corresponding rainbow perfect matching as follows. 
Let $i =\phi(U_{k}^{r})$ and $j = \phi(U_{\ell}^{r})$. 
Since every vertex $z\in U_{\ell}^r$ has exactly one neighbour $y \in U_{k}^{r}$, and there is exactly one edge $e$ in $M_{ij}$ that contains $\psi(y)$, then we can use this edge to uniquely determine $\psi(z)$ by choosing $\psi(z)$ to be the vertex in $e\setminus \{\psi(y)\}$. Note by definition of $M_{ij}$ that this ensures that $ \psi(z)\in I_j$, that distinct vertices in $U_{\ell}^r$ are assigned distinct labels, and that colours of all edges added between $U_{\ell}^r$ and $U_{k}^{r}$ are distinct and belong to $C_{ij}$. This completes our embedding $\psi$ of $T$.

It remains to check that what we have embedded altogether is rainbow and that we have not used too many labels. 
    
\begin{claim}\label{claim:C_ijdisjoint}
    Let $i,j,i',j' \in [\eta] \cup \{0\}$ be such that $|j-i| \neq |j'-i'|$. Then $C_{ij} \cap C_{i'j'} = \emptyset$.
    \end{claim}

    \begin{proofclaim}
 Without loss of generality assume $j > i$ and $j' > i'$ and $j-i < j'-i'$. In particular since these are integer valued, we know $j-i \leq j' - i' -1$. For any $c \in C_{ij}$, it follows that 
$$ c \leq(j-i)(d+3|S|)+\left\lfloor\frac{d+3|S|}{2}\right\rfloor \leq (j'-i' -1)(d+3|S|) +\left\lfloor\frac{d+3|S|}{2}\right\rfloor =(j' - i')(d+3|S|) - \left\lceil \frac{d+3|S|}{2} \right \rceil$$
so $c \notin C_{i'j'}$. Thus there is no $c$ such that $c \in C_{ij} \cap C_{i'j'}$, as desired.
\end{proofclaim}

\begin{claim}\label{claim:disjoint3J}
If $i,j \in [\eta] \cup \{0\}$ are distinct and $|j-i| \neq 1$, then $C_{ij} \cap C_S = \emptyset$.
\end{claim}
\begin{proofclaim}
Without loss of generality suppose $j>i$, and we can assume $j-i \geq 2$ since $i$ and $j$ take integer values, and $j-i \neq 1$. For any $c \in C_{ij}$, we have $c \geq 2(d+3|S|) - \left\lfloor\frac{d+3|S|}{2}\right\rfloor > 3|S|$. Since $C_S = [3|S|-1]$, then $c \notin C_S$ and the statement holds.
\end{proofclaim}

Our labelling of $\taux$ is rainbow so there are no $i,j,i',j' \in [\eta]\cup \{0\}$ and edges $U_k^rU_{\ell}^r,  U_{k'}^s U_{\ell'}^s \in E(T_{\text{aux}})$ for which $i = \phi(U_{k}^r)$, $j = \phi(U_{\ell}^r)$, $i' = \phi(U_{k'}^r)$, $j' = \phi(U_{\ell'}^r)$ and $|j-i|=|\ell - k|$. Recall also that this embedding of $\taux$ avoided colour $1$, and so similarly there are no such $i$ and $j$ with $|j-i|=1$. 
Whilst embedding $S$, we used colour set $C_S$. Whilst embedding root vertices, we used colour sets of the form $C_{0j}$ where $j = \phi(U_{\ell}^r)$ and $vU_{\ell}^r\in E(\taux)$ for some $\ell \in [|F|]$ and $r\in [\var]$. Whilst embedding all remaining vertices, we used colour sets of the form $C_{ij}$ where $i=\phi(U_{k}^r)$, $j = \phi(U_{\ell}^r)$ and $U_{k}^rU_{\ell}^r\in E(\taux)$ for some $k,\ell \in [|F|]$ and $r\in [\var]$. Thus together Claims \ref{claim:C_ijdisjoint} and \ref{claim:disjoint3J} tell us that the family $\mathcal{C}$ of all colour sets considered are pairwise disjoint. By construction of our labelling of $T$, we use each colour set $C\in \mathcal{C}$ at most once, and under $\psi$, the edges with a colour in $C$ are rainbow. Altogether, we deduce that the 
resulting embedding of $T$ is in fact rainbow, as desired.

Finally let us count the total the number of labels used. 
We know that $d\leq \gamma m +2$ and $|S|\leq \lambda m$. So, using the hierarchy \cref{hierarchyTnotJ} we have
\begin{align*}
    \tilde{m} = (\eta+1)(d+3|S|) &\leq ((1+\varepsilon/2)\gamma^{-1}+1)(3\lambda m+\gamma m +2)\\
    &\leq \left(1+\varepsilon/2+\gamma +((1+\varepsilon/2)\gamma^{-1}+1)(3\lambda +2/m)\right)m\\
    & \leq (1+\varepsilon)m,
\end{align*}
as desired.
\end{proof}

\section[Proof of the main lemma]{Proof of \cref{maintheorem}}\label{subsection:PROOF}

We can now finally collate what we know to prove \cref{maintheorem}, 
following the proof strategy given in~\cref{strategy}. As discussed there, we need to be careful with how we choose our parameters in order to ensure that $|S_{\text{high}}|$ is small and all waste vertices in $W$ have low degree. So let us add in a brief sketch of how we may select these.

At the beginning, we consider many possible choices for a parameter $\Delta$, chosen so that $\Delta_0 \ll \Delta_1 \ll \ldots \ll \Delta_{4\varepsilon^{-1}} \ll n$ and for a given $n$-vertex tree $T$, there exist two consecutive $\Delta_i, \Delta_{i+1}$ satisfying the following property: there are at most $\varepsilon n/2$ edges in $T$ which have an endvertex amongst the set of vertices with degree in the interval $[\Delta_i, \Delta_{i+1})$. Call this `special' set of vertices $\tilde{V}$, and choose $S_{\text{high}}$ to be the set of vertices with degree at least $\Delta_{i+1}$. So, $|S_{\text{high}}| \leq 2n/\Delta_{i+1}$.  Furthermore, the waste vertices in $W$ either have degree strictly less than $\Delta_i$, or, they belong to $\Tilde{V}$. In the former case we can show that the number of edges touching this set is at most $\Delta_i|W| \leq \varepsilon n /2$, and in the latter case $\tilde{V}$ was chosen specially so that there are also at most $\varepsilon n /2$ touching this set. Therefore, when we embed $W$ arbitrarily at the end, the number of edges that do not receive a distinct colour is at most $\varepsilon n$.

\begin{proof}[Proof of \cref{maintheorem}]
Let $\eps>0$ and let $\Delta_0, \Delta_1, \dots, \Delta_{4\varepsilon^{-1}}$ be a sequence of natural numbers chosen to satisfy 
\begin{equation*}
     \Delta_{4\varepsilon^{-1}}^{-1} \ll \ldots \ll  \Delta_1^{-1} \ll \Delta^{-1}_0 \ll \eps. 
\end{equation*}
Choose $N$ to be sufficiently large with respect to the $\Delta_i$ and let $n>N$.
Let $T$ be an $n$-vertex tree.  For each $i \in \{0,1\ldots, 4\varepsilon^{-1} -1\}$, let $U_i \coloneqq \{v\in V(T): d_T(v) \in [\Delta_i,\Delta_{i+1})\}$ and let $E_i \subseteq E(T)$ be the set of edges containing a vertex in $U_i$. Each edge in $E(T)$ belongs to at most two distinct $E_i$s. Suppose $|E_i| > \frac{\varepsilon n}{2}$ for every $i\in \{0,1\ldots, 4\varepsilon^{-1} -1\}$, then
$$
2|E(T)| \geq \sum_i|E_i|>4\varepsilon^{-1}\left(\frac{\varepsilon n }{2}\right) = 2n,
$$
a contradiction. So there exists an $i\in \{0,1\ldots, 4\varepsilon^{-1} -1\}$ such that $|E_i| \leq \frac{\varepsilon n}{2}$ and for such an $i$, for convenience let us define $\Delta \coloneqq \Delta_i$, $\tilde{\Delta}\coloneqq \Delta_{i+1}$ and $U\coloneqq U_i$, noting for later that at most $\eps n/2$ edges of $T$ contain a vertex in $U$, and that every vertex in $U$ has degree less than $\tilde{\Delta}$ in $T$.

We now additionally choose $\delta,\zeta>0$ such that $\zeta$ is at most the output $\zeta_0$ of \cref{lem:structure4} when applied with $\delta$, such that $\zeta n\in \N$, and satisfying
\begin{equation}\label{eq:hierarchytheorem}
    n^{-1}\ll \Tilde{\Delta}^{-1}\ll \zeta \ll \delta \ll \Delta^{-1} \ll \eps.
\end{equation}
It is easily observed that $\delta \Delta <\eps/2$ and $\delta \tilde{\Delta}>20$. 

Let $S_{\text{high}} \coloneqq \{v\in V(T):d_T(v)\geq \tilde{\Delta}\}$ so that $|S_{\text{high}}|\leq 2n/\tilde{\Delta} < \delta n/10$. Since we chose $\zeta$ accordingly, we can apply \cref{lem:structure4} to $T$ with $S_{\text{high}}$ playing the role of $S$ to obtain a vertex set $W \subseteq V(T)\setminus S_{\text{high}}$ and a forest $F$ with rooted trees as components, such that $|W|\leq \delta n$ and $T\setminus (W\cup S_{\text{high}})\cong F \times \zeta n$, and for every component in $F \times \zeta n$, only the root has a neighbour in $S_{\text{high}}$. Let $\tilde{T} = T \setminus W$, so that $\tilde{T}$ is a forest of the form  $S_{\text{high}} \cup (F\times \zeta n)$, and we know that only the roots in $F \times \zeta n$ have a neighbour in $S_{\text{high}}$. Note that $ \tilde{n} \coloneqq |\tilde{T}| = |T|-|W| \geq (1-\delta)n$. Choose $\tilde{\zeta}\coloneqq \zeta n /\tilde{n} $, so we have $\tilde{\zeta} \in [\zeta,\frac{\zeta}{1-\delta}]$ and $\tilde{\zeta} \ll \varepsilon$, and $\tilde{T} \setminus S_{\text{high}} \cong F\times \zeta n = F\times \Tilde{\zeta}\tilde{n}$. Observe that $|S_{\text{high}}|\leq \frac{2n}{\tilde{\Delta}}\leq \frac{2\tilde{n}}{\tilde{\Delta}(1-\delta)}$. From~\cref{eq:hierarchytheorem} we can assume 
$ \frac{2}{\tilde{\Delta}(1-\delta)} \ll \zeta \leq  \Tilde{\zeta}$. Applying \cref{lem:TminusJrooted_comps} with  $\tilde{T}$, $\tilde{n}$, $S_{\text{high}}$,  $\frac{2}{\tilde{\Delta}(1-\delta)}$ and $\tilde{\zeta}$ playing the roles of $T$, $m$, $S$, $\lambda$ and $\zeta$ respectively, we find a rainbow copy of $\tilde{T}$ in $K_{[(1+\varepsilon)\tilde{n}]} \subseteq K_{[(1+\varepsilon)n]}$.

All that is left to embed of $T$ is the set $W$. First, let us embed $W \setminus U.$ Since $W \setminus U \subseteq V(T) \setminus (S_{\text{high}}\cup U)$, then every vertex in this set has degree at most $\Delta$. So, the number of edges containing a vertex in $W\setminus U$ is at most $|W \setminus U|\Delta \leq |W|\Delta \leq \delta \Delta n < \varepsilon n /2$. Therefore, we can arbitrarily embed vertices in $W$ into $K_{[(1+\eps)n]}$, ensuring only unused labels of $[(1+\varepsilon)n]$ are used, causing at most $\varepsilon n /2$ colours to repeat. Finally we must embed $W \cap U$.  
We can similarly give these vertices an arbitrary label from what remains of $[(1+\varepsilon)n]$. This adds at most $\varepsilon n/2$ edges, each of which may cause a colour to be repeated. Altogether we find a copy of $T$ in $K_{[(1+\varepsilon)n]}$ such that all but at most $\varepsilon n/2 + \varepsilon n/2 = \varepsilon n$ edges have distinct colours. This proves the lemma.
\end{proof}

\section*{Acknowledgements}

	We would like to thank the two anonymous referees for their helpful comments.

\setlength{\baselineskip}{13pt}
\providecommand{\noopsort}[1]{}

\end{document}